\documentclass{article}
\usepackage{amsmath,amssymb,enumitem,multicol,amsthm,xcolor,hyperref}

\usepackage{enumitem}
\setlistdepth{9}

\newcommand{\begincases}{\begin{enumerate}[label*={\sc Case \arabic*},wide, labelwidth=!, labelindent=0pt]}

\newcommand{\begincasesa}{\begin{enumerate}[label={\sc Case ({\rm \roman*})},wide, labelwidth=!, labelindent=0pt]}
\newcommand{\beginsubcases}{\begin{enumerate}[label*= {\rm .\arabic*},wide, labelwidth=!, labelindent=0pt]}

\newlist{Cases}{enumerate}{9}
\setlist[Cases,1]{label={\sc Case {\rm \arabic*}},wide, labelwidth=!, labelindent=0pt}
\setlist[Cases,2]{label*= {\rm .\arabic*},wide, labelwidth=!, labelindent=0pt}
\setlist[Cases,3]{label*= {\rm .\arabic*},wide, labelwidth=!, labelindent=0pt}
\setlist[Cases,4]{label*= {\rm .\arabic*},wide, labelwidth=!, labelindent=0pt}
\setlist[Cases,5]{label*= {\rm .\arabic*},wide, labelwidth=!, labelindent=0pt}
\setlist[Cases,6]{label*= {\rm .\arabic*},wide, labelwidth=!, labelindent=0pt}
\setlist[Cases,7]{label*= {\rm .\arabic*},wide, labelwidth=!, labelindent=0pt}
\setlist[Cases,8]{label*= {\rm .\arabic*},wide, labelwidth=!, labelindent=0pt}
\setlist[Cases,9]{label*= {\rm .\arabic*},wide, labelwidth=!, labelindent=0pt}

\newlist{CasesA}{enumerate}{1}
\setlist[CasesA,1]{label={\sc Case {\rm (\alph*)}},wide, labelwidth=!, labelindent=0pt}

\newcommand{\sx}[2]{S_{#1}(#2)}
\newcommand{\nt}{\sigma}
\newcommand{\gt}{\Theta}
\newcommand{\constr}{\hat \ve_{\Om+1} }
\newcommand{\tv}[1]{(#1)_\vartheta}
\newcommand{\A}{\nt}
\newcommand{\B}{\nt}

\newcommand{\fs}[2]{ #1 [#2]}
\newcommand{\cfs}[2]{ #1 \{ #2\}}
\newcommand{\nfs}[2]{ \cfs #1   #2 }

\newcommand{\beginclaims}{\begin{enumerate}[label*={\sc Claim },wide, labelwidth=!, labelindent=0pt]}

\newcommand{\coeffs}[1]{{\rm C}(#1)}

\newcommand{\mc}[1]{ #1^* }

\newcommand{\mck}[2]{ #2^*}

\newcommand{\putaway}[1]{}

\newcommand{\ve}{\varepsilon}

\newcommand{\Om}{{\Omega}}

\newcommand{\al}{{\alpha}}

\newcommand{\be}{{\beta}}

\newcommand{\ga}{\gamma}

\newcommand{\de}{\delta}

\newcommand{\la}{\lambda}

\newcommand{\om}{\omega}

\newcommand{\vt}{{\vartheta}}

\newtheorem{theorem}{Theorem}[section]
\newtheorem{lemma}[theorem]{Lemma}
\newtheorem{claim}[theorem]{Claim}
\newtheorem{remark}[theorem]{Remark}
\newtheorem{definition}[theorem]{Definition}
\newtheorem{corollary}[theorem]{Corollary}

\newtheorem{proposition}[theorem]{Proposition}
\newtheorem{example}[theorem]{Example}

\begin{document}

\title{Fundamental sequences and fast-growing hierarchies for the Bachmann-Howard ordinal}
\author{David Fern\'andez-Duque and Andreas Weiermann}
\maketitle

\begin{abstract}
Hardy functions are defined by transfinite recursion and provide upper bounds for the growth rate of the provably total computable functions in various formal theories, making them an essential ingredient in many proofs of independence.
Their definition is contingent on a choice of fundamental sequences, which approximate limits in a `canonical' way.
In order to ensure that these functions behave as expected, including the aforementioned unprovability results, these fundamental sequences must enjoy certain regularity properties.

In this article, we prove that Buchholz's system of fundamental sequences for the $\vartheta$ function enjoys such conditions, including the Bachmann property.
We partially extend these results to variants of the $\vartheta$ function, including a version without addition for countable ordinals.
We conclude that the Hardy functions based on these notation systems enjoy natural monotonicity properties and majorize all functions defined by primitive recursion along $\vartheta(\ve_{\Om+1})$.
\end{abstract}

\section{Introduction}

It is an empirically observed phenomenon that natural arithmetical theories are well-ordered by consistency strength.
This makes it possible to measure said strength using ordinal numbers, a key objective in contemporary proof theory \cite{RathjenSynthese}.
Given a theory $T$, its proof-theoretic ordinal, which we may denote $\|T\|$, provides a wealth of information about its arithmetical consequences; in particular, $\|T\|$ can be used to characterize the provably total computable functions of $T$ in terms of transfinite recursion \cite{Buchholz91,FriedmanS95}.

To be precise, one can define a hierarchy $(H_\al)_{\al<\|T\|}$ of increasingly fast-growing functions bounding all provably total computable functions of $T$.
The precise growth rate of such functions depends on a system of {\em fundamental sequences,} which to each limit $\al<\|T\|$ assigns an increasing sequence $(\cfs\al n)_{n<\om}$ with $\al$ as its limit.
Fundamental sequences allow us to define recursion along $\al$ computably and, if these fundamental sequences satisfy some regularity properties, they indeed majorize the provably total computable functions of $T$ \cite{BCW}.
Fundamental sequences satisfying said regularity conditions are known as {\em regular Bachmann systems.}

In this sense, the ordinal
\[\varepsilon_0 =\sup \{\om,\om^\om,\om^{\om^\om}\ldots\}, \]
measures the proof-theoretic strength of Peano arithmetic \cite{Gentzen1936}.
Ordinals below $\varepsilon_0$ can be represented in terms of $0,+$ and the function $\xi\mapsto\om^\xi$.
If we wish to represent the proof-theoretic ordinal of stronger theories, we need a richer set of ordinal functions.
From a methodological perspective, it is convenient to have a rich enough system of ordinal notations to accommodate many familiar theories within a uniform scale.
A reasonable upper bound for the characterization of theories that are `not terribly strong' is the Bachmann-Howard ordinal, sometimes denoted $\psi(\ve_{\Om+1})$ \cite{Pohlers1981}, where $\Om$ is the first uncountable ordinal and $\psi$ represents its countable collapse.

In this paper we will instead represent this ordinal in terms of the $\vartheta$ function, originally introduced by Rathjen~\cite{RathjenFragments,RathjenKruskal}.
Buchholz has defined a system of fundamental sequences based on this notation \cite{BuchholzOrd}, and has also contemplated variants of the $\vartheta$ function relative to some set $\mathbb X$ club in $\Om$ \cite{BuchholzSurvey}.
In this paper we show that Buchholz's fundamental sequences indeed form a regular Bachmann system, and thus the associated fast-growing hierarchies can be used to bound the provably total computable functions of theories of strength up to and including the Bachmann-Howard ordinal.
We will present results in terms of arbitrary $\mathbb X$ when possible, but will focus especially on the cases when $\mathbb X$ is the class of additively indecomposable ordinals, giving rise to the standard $\vartheta$ function, as well as the case where $\mathbb X=1+\sf Ord$, which gives rise to a `slower' version we denote $\nt$.
As we will see, $\nt$ provides notations for the same ordinals as $\vartheta$, with the feature that addition is not needed for countable ordinals.
This may at first be surprising since addition-free systems are usually weaker (see e.g.~\cite{RathjenAddition}), but notations based on $\sigma$ do allow for addition between {\em uncountable} ordinals, avoiding this loss of strength.
We remark that there are other natural choices for $\mathbb X$, which we do not explore in detail; one canonical choice is to take $\mathbb X$ to be the class of $\varepsilon$-numbers, i.e.~ordinals $\alpha$ such that $\alpha=\omega^\alpha$, suitable for ordinal notation systems with $\xi \mapsto \omega^\xi$ as a primitive operation.

The layout of the paper is as follows.
Section~\ref{secOrd} reviews notations and fundamental sequences for ordinals below $\varepsilon_{\Omega+1}$ and Section~\ref{secParTheta} defines parametrized theta functions.
Section~\ref{secXLim} shows how to approximate values of $\Theta_\mathbb X$, setting the stage for defining fundamental sequences later.
Section~\ref{secRegularBS} discusses notions from~\cite{BCW} involving norms and fundamental sequences for countable ordinals, arriving at what we call {\em regular Bachmann systems.}
Section~\ref{secFS} then discusses Buchholz's fundamental sequences in the general setting of parametrized theta functions.
The results in Sections~\ref{secOrd}, \ref{secParTheta}, \ref{secXLim} and \ref{secFS} have appeared in \cite{BuchholzOrd,BuchholzSurvey} and are reproduced with Buchholz's permission.
Section~\ref{secThetaFun} compares the $\vartheta$ and $\nt$ functions and shows how to express one function in terms of the other, while Section~\ref{secBachmannP} shows that Buchholz's systems of fundamental sequences for both $\vartheta$ and $\sigma$ enjoy the Bachmann property.
The proofs in this section can be adapted to other parametrized theta functions by considering any extra cases that arise (e.g.~for the function $\xi\mapsto\omega^\xi$) assuming, of course, that the Bachmann property indeed holds.
Section~\ref{secNorms} discusses norms for ordinal notation systems in some generality and then applies them to $\vartheta$ and $\sigma$ to show that both are regular Bachmann systems (in fact, $\vartheta$ is {\em Cantorian}, which means that it treats $\omega$ exponentiation in the `standard' way), which is our main technical result.
Section~\ref{sectHardy} reviews Hardy functions and establishes, using results of~\cite{BCW}, that every primitive recursive function relative to the Bachmann-Howard ordinal is dominated by both the Hardy function based on $\vartheta$ and on $\sigma$.
Section~\ref{secConc} gives some concluding remarks.

We aim for an accessible and self-contained presentation, so we assume familiarity only with elementary ordinal arithmetic.

\section{Ordinals below $\ve_{\Om+1}$}\label{secOrd}


In this section, we review some basic notions from ordinal arithmetic.
We assume familiarity with ordinal addition, multiplication, and exponentiation.
The class of successor ordinals (those of the form $\eta+1$) will be denoted $\rm Succ$, while the class of limit ordinals is denoted $\rm Lim$. 
The predecessor of $\al$ will be denoted $\al-1$, when it exists.

Throughout the text, we will use normal forms for ordinals based on either $\om$ or $\Om$, the first uncountable ordinal.
Let $\kappa\in \{\om,\Om\}$ and $\xi>0$ be ordinals.
There exist unique ordinals $\al,\be,\ga$ with $\be<\kappa$ such that $ \xi=\kappa^\al\be+\ga$ and $\ga<\kappa^\al$.
This is the {\em $\kappa$-normal form of $\xi$.}
If $\kappa=\om$, then $\be$ is a finite number and we can replace the term $\om^\al\be$ by $\om^\al+\ldots +\om^\al$ ($\be$ times), leading to the usual Cantor normal form.
More generally, we say that an expression $\al+\be$ is in {\em Cantor normal form} if $\al+\be>\al'+\be $ for all $\al'<\al$.

\begin{remark}
In our presentation, $\al$, $\be$ and $\ga$ are regarded as ordinals rather than terms; formally, $\Om^\al\be+\ga$ should be regarded as a triple of ordinals $(\al,\be,\ga)$. Thus we do not include the standard clause that they should be inductively written in normal form.
We will work with ordinals rather than notation systems when possible, although the latter will be revisited in Section~\ref{secNorms}.
\end{remark}

The ordinal $\ve_{\Om+1}$ is defined as the least $\ve>\Omega$ such that $\ve=\om^\ve$.
If $\xi=\Om ^\al\be+\ga <\ve_{\Om+1}$ is in $\Om$-normal form, then we may also deduce that $\al<\xi$.
We may define a general version of fundamental sequences for ordinals in $\Om$-normal form, where each $\xi<\ve_{\Om+1}$ is approximated by a sequence of suitable cofinality.
These general fundamental sequences will be useful later to provide more standard fundamental sequences for countable ordinals.

\begin{definition}
We define fundamental sequences for ordinals below $\ve_{\Om+1}$ and $\theta<\Om$ recursively as follows.
\begin{enumerate}

\item $\fs 0\theta = \fs 1\theta = 0$,

\item $ \fs{( \Omega^\alpha\be + \ga )}{\theta} =  \Omega^\alpha\be + \fs{ \ga }{\theta} $ if $\Omega^\alpha\be + \ga$ is in $\Omega$-normal form and $\ga>0$,

\item $ \fs{( \Omega^\alpha\be  ) }{\theta} =  \Omega^\alpha\theta $ if $\be$ is a limit,

\item $ \fs{ \Omega^{\al+1}   }{\theta} =  \Omega^\al \theta $,

\item $ \fs{ ( \Omega^\alpha(\be+1) )  }{\theta} =  \Omega^\al \be +  {\fs {\Omega^\al}\theta } $ if $\be>0$, and

\item $ \fs{ \Omega^\alpha }{\theta} =  \Omega^{\fs \al\theta } $ if $\al$ is a limit.

\end{enumerate}
\end{definition}

Note that $\al+\be$ being in Cantor normal form does not suffice to guarantee $\fs{(\al+\be)}\theta=\al+\fs\be\theta$, as for example $\fs{\Om(\om+\om)}1=\Om\neq \Om\om+\fs{(\Om\om)}1$.
However, we do have a restricted version of this phenomenon.
The proof follows a simple induction and is left to the reader.

\begin{lemma}\label{lemmFSSum}
Suppose that $\Om^\al(\be+\de)+\ga <\ve_{\Om+1}$ is in $\Om$-normal form, $\de$ is a successor or zero, and $\Om^\al \de +\ga>0$.
Then, for all $\theta$,  $\fs{(\Om^\al(\be+\de)+\ga )}\theta=\Om^\al\be+\fs{(\Om^\al \de +\ga )}\theta$.
\end{lemma}

Theta functions are defined in terms of the maximal coefficient of an ordinal.
We define the set of coefficients of $\xi <\varepsilon_{\Om+1}$ by $\coeffs  0 = \{0\}$ and $\coeffs  {\Omega^{\alpha}\be+\ga} = \coeffs  {\al} \cup \coeffs  \ga \cup \{\be\}$ if $\Omega^{\alpha}\be+\ga$ is in $\Om$-normal form, and the maximal coefficient of $\alpha$ by $\mck\Omega\al=\max\coeffs\al$.
For $\xi<\varepsilon_{\Omega+1}$, the {\em terminal part} of $\xi$, denoted $\tau(\xi)$, is given recursively by
\begin{enumerate}

\item $\tau(0) = 0$ and $\tau(\zeta+1) = 1$,

\item $\tau (\Omega^\al \be+\ga) = \tau  (\ga)$ if $\ga $ is a limit and $\Omega^\al \be+\ga$ is in $\Omega$-normal form,

\item $\tau (\Omega^\al \be ) = \be$ if $\be$ is a limit,

\item $\tau (\Omega^\al (\be+1) ) = \tau(\al)$ if $\al$ is a limit, and

\item $\tau(\Omega^ {\al+1}(\be+1)) = \Omega$.

\end{enumerate}
For notational convenience, we have defined $\fs\xi\theta$ even when $\theta\geq \tau$, although normally one is interested in cases where $\theta<\tau$.

\begin{example}
Let $\xi =\Om^{ \omega^3+\om}+\Om^{ \om^2 }\om$, written in $\Om$-normal form.
We can see that $\coeffs \xi = \{0, \omega^3+\om,\om^2,\om \}$, so that $\xi^* = \om^3+\om$, as this is the largest element of $\coeffs \xi$.
(To compute e.g.~$\coeffs{\omega^3+\om}$, we must write it as $\Om^0(\omega^3+\om)+0$ in $\Om$-normal form and then apply the definition of $\coeffs\cdot$.)
Meanwhile, $\tau(\xi) = \om$, as it is the rightmost limit coefficient.

If instead we consider $\xi'=\xi+\Om^{ \om^2 }$, we note that $\Om^{ \omega^3+\om}+\Om^{ \om^2 }\om+\Om^{ \om^2 }$ is in $\Om$-normal form according to our definitions with $\al= \omega^3+\om $, $\be=1$ and $\ga = \Om^{ \om^2 }\om+\Om^{ \om^2 }$; note however that $\ga$ is not written in normal form, so we may instead write $\ga= \Om^{ \om^2+\om }(\om+1)$ to obtain $\xi'=\Om^{ \omega^3+\om }+\Om^{ \om^2+\om }(\om+1)$.
Now we see that $\tau(\xi')=\om^2+\om$ is its terminal part since the coefficient $\om+1$ is a successor.

It is readily checked that $\xi =\lim_{n\to\infty} \Om^{ \omega^3+\om}+\Om^{ \om^2 }n$, or in other words, $\xi=\lim_{n\to\tau(\xi)} \fs\xi n$.
Similarly, $\xi'=\lim_{n\to\infty} \Om^{ \omega^3+\om}+\Om^{ \om^2 }\om+\Om^{ \om n }$.
\end{example}

Below and throughout the text, if $(\al_\theta)_{\theta<\tau}$ is any sequence, we write `$\al_\theta \nearrow\al$ as $\theta\nearrow\tau$' to indicate that $\al_\theta$ is an increasing sequence converging to $\al$.

\begin{lemma}[\cite{BuchholzOrd}]\label{lemmBoundMC}
Let $\la <\varepsilon_{\Omega+1}$ be a limit ordinal.
Then,
\begin{enumerate}

\item $\tau = \tau(\la )$ is a limit and $\la=\fs\la \tau$;

\item\label{itBoundMCConverge} $\fs\la \theta\nearrow \la $ as $\theta\nearrow\tau$;

\item\label{itBoundMCTwo}
if $\theta<\Om$, then $\theta\leq \mc {\fs\la \theta} \leq \max \{\mc \la ,\theta  \}$;

\item\label{itFundPropMajor} if $\xi < \lambda$ and $\mck\Omega \xi < \mc{\fs\lambda \theta} $, then $\xi < \fs\la\theta $.

\end{enumerate}
\end{lemma}

\begin{proof}
We prove the last item and leave the others to the reader.
Proceed by induction on $\lambda = \Om^\alpha\beta+\eta$ in $\Om$-normal form.
We prove the contrapositive and assume that $\fs\la\theta\leq \xi <\la$ and show that $\mc\xi\geq \mc{\fs \la\theta}$.
\begin{Cases}
\item ($\eta>0$).
From
\[\Om^\al\be+\fs\eta \theta = \fs \lambda\theta \leq \xi <\la  \]
we obtain $\xi = \Om^\al\be +\de$ with $\fs \eta \theta \leq \de<\eta$.
The induction hypothesis yields $\de^*\geq \mc{\fs\eta\theta} $, from which it readily follows that $\mc\xi \geq \mc{\fs\lambda\theta}$.

\item ($\eta=0$).
This case divides into several sub-cases.
\begin{Cases}

\item ($\be\in \rm Lim$).
Then $\fs \lambda\theta = \Om^\al\theta$, and since $\Om^\al\theta\leq \xi <\Om^\al\beta$ we obtain $\theta<\beta$ and $\xi=\Om^\al\rho+\delta$ for some $\rho\geq\theta$ and $\de\geq 0$.
Then $\mc\xi\geq \mc{(\Om^\al\rho)}\geq \mc{(\Om^\al\theta)} =\mc{\fs\la\theta}$.

\item ($\be = \ga+1$).
Consider further sub-cases according to the shape of $\al$.
\begin{Cases}
\item ($\al=0$).
Then $\fs\la\theta=\ga\leq \xi<\la$, so $\xi=\ga = \mc{\fs\la\theta}$ and $\xi^*=\xi=\mc{\fs\la\theta}$.

\item ($\al=\chi+1$).
We have that
\[\Om^{\chi+1}\ga+\Om^\chi\theta  = \fs \lambda\theta \leq \xi <\la . \]
It follows that $\xi=\Om^{\chi+1}\ga+\Om^\chi \rho+\de$ with $ \de<\Om^\chi$ and $\rho\geq\theta$, hence by inspection  $\mc\xi\geq \mc{\fs\la\theta}$. 

\item ($\al\in\rm Lim$).
Then
\[\Om^{\al}\ga+\Om^{\fs\al\theta}  = \fs \lambda\theta \leq \xi <\la . \]
We then have that $\xi = \Om^{\al}\ga+\Om^{\zeta}\rho+\delta$ with $ \fs\al\theta\leq \zeta<\al$, $\de<\Om^{\zeta}$, and $\rho>0$.
The induction hypothesis yields $\mc\zeta\geq\mc{\fs\al\theta}$, from which it can be checked that all coefficients of $\fs\la\theta$ are bounded by coefficients of $\xi$, i.e.~$\mc\xi\geq \mc{\fs\la\theta}$. \qedhere
\end{Cases}
\end{Cases}
\end{Cases}
\end{proof}

We often consider ordinals in the form $\alpha+\beta$, where $\alpha=\Om\tilde\alpha$ and $\beta<\Om$.
Sometimes it will be useful in these cases to apply fundamental sequences directly to $\al$.
In this setting, it is convenient to know that the end result maintains the form $\Om\tilde\ga+\be$.
The following makes this precise.

\begin{lemma}\label{lemmOmAlph}
If $0 < \al <\ve_{\Om+1}$ and $\tau(\Om\al)<\Om$ then for every $\theta>0$ we have that $\fs{(\Om\al)}\theta = \Om\ga$ for some $\ga$.
\end{lemma}

\begin{proof}
By induction on $\al$.
Write $\al=\Om^{\eta}\be+\ga$ in $\Om$-normal form and note that $\eta>0$ and $\ga=\Om\tilde\ga$.
If $\ga>0$, we may apply the induction hypothesis to $\ga$, so assume otherwise.
If $\be$ is a limit then $\fs{(\Om\al)}\theta = \Om^{\eta}\theta $, which is of the required form since $\eta>0$.
If $\be = \de+1$ then $\fs{(\Om\al)}\theta = \Om^\eta\de+ \Om^{\fs\eta\theta} $.
Moreover, note that in this case $\eta$ must be a limit since otherwise we would obtain $\tau(\Om\al)=\Om$.
Hence $\fs\eta\theta>0$, and again we see that $\fs{(\Om\al)}\theta$ is of the required form.
\end{proof}

On occasion we want to `remove' the terminal part of $\xi$ without altering $\xi$ itself too much.
We can use fundamental sequences for this end, since they affect only the rightmost part of $\xi$.
For example, let $\xi=\Om^{\Om^{\om^2}\om}$.
We have that $\xi^*=\om^2$ while $\tau(\xi)=\om$.
If we compute $\fs\xi 0$ we obtain $\fs\xi 0=\Om^0$, so indeed we have removed the terminal part but we also lost the larger maximal coefficient.
If instead we try $\fs\xi 1$, we now obtain $\fs\xi 1=\Om^{\Om^{\om^2}1}$, which maintains more of the structure of $\xi$; in particular, the maximal coefficient has not been lost.
Thus we will follow Buchholz~\cite{BuchholzOrd} in working with $\xi\mapsto \fs \xi 1$ when we do need to remove the terminal part of $\xi$.
Despite our example, it is in fact possible to have $\fs \xi 1^*<\xi^*$, but only in very specific cases.

\begin{lemma}\label{lemmStarPlusOne}
Let $\xi<\ve_{\Om+1}$ and suppose that $\fs\xi 1^*<  \xi^*$.
Then, either $\xi^*=\tau(\xi)\in \rm Lim$, or $\xi^*= \fs\xi 1^*+1$.
\end{lemma}

\begin{proof}
The argument goes by induction on $\xi$.
The case for $\xi=0$ is vacuously true since $\fs\xi 1^*< \xi^*$ is impossible.
If $\xi=\Om^\al\be+\ga$ with either $\ga>0$ or $\ga=0$ and $\be=1$, we may apply the induction hypothesis to $\ga$ or $\al$, respectively.
If $\xi=\be+1$ with $\be$ countable, then $\xi^*=\be+1$ and $\fs\xi 1 = \be$, so $\fs\xi 1^*=\be$ as well.
If $\xi=\Om^\al\be$ with $\be$ a limit, then $\be=\tau(\xi)$ and $\fs\xi 1 =\Om^\al $.
Thus $\coeffs\al \subseteq \coeffs{\fs\xi 1}$, so the only way to have $\fs\xi 1^*<\xi^*$ is if $\xi^*=\be =\tau(\xi)$.
Otherwise, $\be=\delta+1$ is a successor and $\fs\xi 1 =\Om^\al\de+\Om^{\fs \al 1}$.
Once again we observe that $\coeffs\al \subseteq \coeffs{\fs\xi 1}$, so the only way to have $\fs\xi 1^*<\xi^*$ is if $\xi^*=\be  $, which leads to $\fs\xi 1^*=\de$.
\end{proof}

As a corollary we obtain that if $\tau(\xi) = \Om$ then $\xi^* \leq \fs\xi 1^*+1$ always holds.

\section{Parametrized theta functions}\label{secParTheta}

The theta function admits a general definition with respect to a class $\mathbb X$ which is assumed club in $\Omega$ and such that $0\notin\mathbb X$~\cite{BuchholzSurvey}.
For $\xi<\ve_{\Om+1}$, $\gt_\mathbb X(\xi)$ is defined recursively by
\[\gt_\mathbb X(\xi) = \min    \big \{\theta \in\mathbb X\cap ( \xi^*,\Om)  :   \forall \zeta<\xi \big (\mck\Om\zeta<\theta \rightarrow \gt_\mathbb X(\zeta)<\theta \big )\big \} .\]

It will be convenient to present this definition in terms of candidates.
Say that $\theta\in \mathbb X$ is a {\em $\gt_\mathbb X$-candidate for $\xi$} if $\theta>\xi^*$ and $\forall \zeta<\xi \big (\mck\Om\zeta<\theta \rightarrow \gt_\mathbb X(\zeta)<\theta \big )$. 
Then, $\gt_\mathbb X(\xi)$ is the least $\gt_\mathbb X$-candidate for $\xi$.

If we let $\mathbb P$ denote the class of {\em principal numbers,} i.e., the ordinals of the form $\om^\xi$, then one of the most common variants in the literature is $\vartheta:=\gt_\mathbb P$.
We also set $\nt := \gt_{1+{\sf Ord}}$, where $1+{\sf Ord} = \{1+\xi:\xi\in {\sf Ord}\}$ or, equivalently, ${\sf Ord}\setminus \{0\}$.

\begin{example}\label{exSigmaSucc}
It will be instructive to compute $\nt(\beta)$ when $\beta<\Om$.
In fact, we claim that in this case, we simply have $\nt(\beta)=\beta+1$.
We may prove this inductively on $\beta$; assuming $\nt(\gamma)=\gamma+1$ for all $\gamma<\beta$, we propose $\beta+1$ as a $\sigma$-candidate for $\beta$.
Indeed, $\beta+1\in 1+\sf Ord$ and if $\gamma<\beta$ and $\gamma^*<\beta+1$ then by the induction hypothesis, $\sigma(\gamma)=\gamma+1<\beta+1$.
Moreover, $\beta+1>\beta=\beta*$ (since the maximal coefficient of a countable ordinal is itself), so $\beta+1$ is indeed a $\sigma$-candidate for $\beta$.
It is obviously least (by the condition $\nt(\beta)>\beta^*$), so $\nt(\beta)=\beta+1$.
Thus $\nt$ can be viewed as an extension of the successor on the countable ordinals (hence the notation), and $\nt[\Om]=\Om\cap \rm Succ$.
\end{example}

Note that since $\mathbb X$ is assumed unbounded, every ordinal $\xi$ has an `$\mathbb X$-successor', which we denote $\sx {\mathbb X}\xi$, defined as the least element of $\mathbb X$ strictly greater than $\xi$.
This will be useful in proving that $\gt_\mathbb X$ is well-defined.

\begin{lemma}[\cite{BuchholzSurvey}]
Given a club $\mathbb X\subseteq(0,\Om)$, the function $\gt_\mathbb X\colon \ve_{\Om+1}\to \Omega$ is well-defined.
\end{lemma}

\begin{proof}
Assume inductively that $\gt_\mathbb X$ is total on $\xi<\ve_{\Om+1}$.
We must show that $\gt_\mathbb X(\xi)$ is defined.
Define a sequence $(\theta_n)_{n<\om}$ by letting $\theta_0$ be the least element of $\mathbb X$ greater than $\xi^* $, which exists by unboundedness of $\mathbb X$, and
\[\theta_{n+1} = \sup\{\sx{\mathbb X}{\gt_\mathbb X(\zeta)} : \zeta<\xi\text{ and }\zeta^*<\theta_n  \}. \]
Let $\theta=\sup_{n<\om} \theta_n$.
First we claim that $\theta$ is countable; to see this, note that it is a countable limit, and moreover each $\theta_n$ is countable since the set $\{\zeta<\xi : \zeta^*<\lambda\}$ is countable for any countable $\lambda$.
Moreover, $\theta\in\mathbb X$ since $\mathbb X$ is assumed closed.
Then note that if $\zeta<\xi$ and $\zeta^*<\theta$, it follows that $\zeta^*<\theta_n $ for some $n$ and so $\gt_\mathbb X(\zeta) <\theta_{n+1}<\theta$.
Thus $\theta$ is $\gt_\mathbb X$-candidate for $\xi$, hence the least candidate exists.\footnote{In fact, $\theta$ is this least candidate.}
\end{proof}

The following can readily be deduced from the definition.

\begin{lemma}\label{lemmThetaOrd}
If $\al<\be<\ve_{\Om+1}$, then $\gt_\mathbb X(\al)<\gt_\mathbb X(\be)$ if and only if $\al^*<\gt_\mathbb X(\be)$.
\end{lemma}

For ordinals $\kappa,\la$, define $\kappa_n(\la)$ recursively by $\kappa_0(\la)=\la$ and $\kappa_{n+1} = \kappa^{\kappa_n(\la)}$.
We write $\kappa_n$ instead of $\kappa_n(1)$.
Note that $\Om_n<\Om_{n+1}$ and $\Om_n^*=\Om_{n+1}^*$ yield $\gt_\mathbb X(\Om_n) < \gt_\mathbb X(\Om_{n+1})$.

\begin{lemma}\label{lemmOmTower}
If $\xi\in \gt_\mathbb X(\Om_n) \cap\mathbb X$, then there is $\zeta < \Om_n$ such that $\xi=\gt_\mathbb X(\zeta)$.
\end{lemma}

\begin{proof}
If $\xi=\min\mathbb X$ then we observe that $\xi = \gt_\mathbb X(0)$.
 Otherwise, let $\zeta$ be least so that $\gt_\mathbb X(\zeta) \geq \xi$ and $\zeta^*<\xi$, which exists and is bounded by $\Om_n$ since $\Om_n$ has these properties.
 We claim that $\xi$ is a $\gt_\mathbb X$-candidate for $\zeta$, which by the assumption $\gt_\mathbb X(\zeta) \geq \xi$  yields $\gt_\mathbb X(\zeta)=\xi$.
 But by our choice of $\zeta$, if $\gamma<\zeta$ is such that $\gamma^*<\xi$, then also $\gt_\mathbb X(\gamma)<\xi$, as needed.
\end{proof}

By abuse of notation, we define $\gt_\mathbb X(\ve_{\Om+1}) = \sup_{n<\om}\gt_\mathbb X(\Om_n)$.

\begin{corollary}\label{corOmTower}
If $\xi \in \gt_\mathbb X(\ve_{\Om+1}) \cap\mathbb X$ there is $\zeta<\ve_{\Om+1}$ with $\gt_\mathbb X(\zeta) = \xi$.
\end{corollary}

We define $\constr$ to be the set of ordinals $\xi<\ve_{\Om+1}$ such that $\xi^*<\gt_\mathbb X(\ve_{\Om+1})$.
Then, $\gt_\mathbb X\colon \constr\to \gt_\mathbb X(\ve_{\Om+1})$ can be checked to be injective and have range $\mathbb X \cap \gt_\mathbb X(\ve_{\Om+1})$ using Lemma \ref{lemmThetaOrd}.

\begin{remark}
Recall that we defined $\nt=\Theta_{1+\sf Ord}$.
In this case, Corollary~\ref{corOmTower} tells us that if $0<\be<\nt(\ve_{\Om+1})$, then $\be=\nt(\zeta)$ for some $\zeta\in\constr$.
Thus any $\Om$-normal form $\Om^\al\be+\ga$ can be rewritten as $\Om^\al\nt(\zeta)+\ga$.
This makes $\nt$ a particularly compelling choice for ordinal notations, as addition does not need to be applied to countable ordinals.
\end{remark}

\section{Approximating $\mathbb X$-limits}\label{secXLim}

The fundamental sequence for an ordinal $\xi$ is a sequence $(\cfs\xi n)_{n<\om}$ such that $\cfs \xi n\nearrow \xi$ as $n\nearrow \om$.
In Section~\ref{secFS}, we will present a system of fundamental sequences due to Buchholz; the goal of this section is to set the stage for this system by showing how to approximate various ordinals according to their `shape'.
Note that outside of $\mathbb X$, the definition of fundamental sequences need not involve the $\gt_\mathbb X$ function; this is also true of points that are isolated in $\mathbb X$, as (by definition) they cannot be approximated by points of $\mathbb X$ itself.
However, for limit points of $\mathbb X$, we can define fundamental sequences with some generality in terms of values of  $\gt_\mathbb X$.
To this end, let $\mathbb X'$ denote the set of limit points of $\mathbb X$; since $\mathbb X$ is assumed closed, we have that $\mathbb X'\subseteq\mathbb X$.

The results in this section are due to Buchholz \cite{BuchholzOrd}, but we include their proofs in order to keep the article self-contained.
Here and in the rest of the paper, we will often write $\gt$ instead of $\gt_\mathbb X$.

Many of the functions used in `predicative' ordinal analysis are {\em normal;} they are increasing and continuous (on at least one of their arguments).
For example, the function $\xi\mapsto\omega^\xi$ is increasing on $\xi$ and has the property that, when $\xi$ is a limit, $\om^\xi=\lim_{\zeta\to \xi}\om^\zeta$.
Normality is very convenient when defining fundamental sequences, as if $f$ is normal then to approximate $f(\lambda)$, it suffices to approximate $\lambda$ when $\lambda$ is a limit.
However, normal functions satisfy $\xi\leq f(\xi)$, so given that $\gt_\mathbb X\colon \constr\to \gt_\mathbb X(\ve_{\Om+1})$, we must have that $\gt_\mathbb X$ is either non-monotone or not (always) continuous.
In fact, both of these are the case.

Let us fix $\mathbb X$ and write $\gt$ for $\gt_\mathbb X$.
Examples of non-monotonicity of $\gt$ are easy to find using Lemma~\ref{lemmThetaOrd}, as for example $\gt (\Om)<\gt (\gt (\Om))$ even though $\gt (\Om)$ is countable.
Identifying the points where continuity fails, or `jumps', is a bit more subtle.
We may count zero and all successors as `jumps', since in these cases it is clear that $\gt   (\xi)$ cannot be obtained by approximating $\xi$, even when $\gt (\xi)$ is a limit ordinal.
If $\xi$ is a limit, then in most cases we {\em will} have that $\gt (\xi)=\lim_{\zeta\to \xi}\gt (\xi)$; for example, $\gt (\Om\om+\om)=\lim_{n\to\omega}\gt (\Om\om+n)$, since any $\zeta<\Om\om+\om$ satisfying $\zeta^*<\gt (\Om\om+\om)$ will either be of the form $\Om\om+n$ or else will satisfy $\zeta < \Om\om$ and, as we will see, also satisfy $\zeta^*<\gt (\Om\om+n)$ for large enough $n$, and hence $\gt(\zeta)<\gt (\Om\om+n)$ for such $n$.
But an issue arises when the terminal part $\tau=\tau(\xi)$ of $\xi$ is also its largest coefficient, leading to instances of $\zeta<\xi$ that have coefficients greater than any $\gt (\fs \xi \eta)$ for $\eta<\tau$.

Consider, for example, $\xi=\Om+\gt (\ga)$, where $\ga=\Om 2$.
The maximal coefficient of $\xi$ is also its terminal part, i.e.
\begin{equation}\label{eqTau}
\xi^*=\tau(\xi),
\end{equation}
and this coefficient not occur elsewhere, which we may capture via the inequality
\begin{equation}\label{eqAst}
\fs\xi 1^*<\xi^*,
\end{equation}
as applying the fundamental sequence at $1$ `kills' the terminal part (see the discussion before Lemma~\ref{lemmStarPlusOne}).
We have that $\gt (\xi)>\gt(\Om+\eta)$ for all $\eta<\gt (\ga)$, but {\em also} have that $\gt (\xi)>\gt(\gt (\ga))$ since $\gt (\ga)=\xi^*<\gt(\xi)$.
Moreover, if $\eta<\gt(\ga)$ then since
\begin{equation}\label{eqGam}
\xi<\ga,
\end{equation}
we have that
$\Om+\eta<\ga$ and $(\Om+\eta)^*=\max\{1,\eta\} <\gt(\ga)$, so $\gt(\eta)<\gt(\ga)$.
This means that $\lim_{\eta\to \gt(\Om 2)} \gt(\Om+\eta) = \gt(\ga) <\gt(\Om+\gt(\ga))$, hence $\gt(\Om+\gt(\ga))\neq \lim_{\eta\to \gt(\ga)} \gt(\Om+\eta)$.
Thus we have identified three conditions, \eqref{eqTau}, \eqref{eqAst}, and \eqref{eqGam}, which lead to $\gt$ being discontinuous at $\xi$.
These conditions lead to the definition of ${\rm FIX}(\mathbb X)$ below.
The intuition is that $\xi\in {\rm FIX}(\mathbb X)$ if $\theta:=\tau$ is `almost' a $\gt$-candidate for $\xi$ -- except of course for the clause $\theta>\xi^*$ -- and hence would be a fixed point of $\theta \mapsto \gt(\fs\al \theta)$ if this clause were removed.

Finally, we note that multiples of $\Om$ are also points of discontinuity.
Suppose that $\al=\Om\tilde\al > 0$ and let $\tau=\tau \al$.
If $\tau<\Om$ then $\tau$ is a limit and, whenever $\tau_n\to \tau$, we have that
\begin{equation}\label{eqTauConv}
\fs \al{\tau_n}+\theta\to \al
\end{equation}
for {\em any} countable $\theta$.
If $\theta$ is too small we might have that $\lim_{n\to \infty}\gt( \fs \al{\tau_n}+\theta) \leq \gt(\al)$, but we may also {\em over}shoot -- say, if $\theta=\gt(\al)$ -- and obtain $\lim_{n\to \infty}\gt ( \fs \al{\tau_n}+\theta) > \gt(\al)$.
As we will see later, there is always a value of $\theta$ that is `just right' (see Definition~\ref{defStar} and Lemma~\ref{lemmThetaTau}).
The case that $\tau=\Om$ is even worse, since no countable sequence can converge to $\Om$ at all, and clearly $\lim_{\eta\to \Om}\gt( \fs \al{\eta} )=\Om > \gt(\al)$.
In this case we instead consider a sequence $\gt(\fs\al{\theta_n})$ where $\theta_n$ converges to $\gt(\al)$ (see Lemma~\ref{lemmThetaOm}).

The following definition classifies these points of discontinuity.
Below, $\Om{\sf Ord}$ is the class of `multiples of $\Om$', i.e.~of ordinals of the form $\Om\al$; note in particular that $0\in \Om\sf Ord$.

\begin{definition}
We define sets
\begin{enumerate}
\item 
${\rm FIX}(\mathbb X) = \{\xi<\ve_{\Om+1}: \fs\xi 1^*<\xi^* =\tau(\xi) = \gt_\mathbb X(\ga)\text{ for some $\ga>\xi$}\}
$

\item 
${\rm JUMP}(\mathbb X) =  {\rm Succ} \cup {\rm FIX}(\mathbb X) \cup \Om{\sf Ord}$.

\end{enumerate}
\end{definition}

\begin{lemma}\label{lemmNotFix}
If $\xi<\Om$ and $\xi\in {\rm JUMP}(\mathbb X) $, then $\gt(\xi)\notin \mathbb X'$.
\end{lemma}

\begin{proof}
Assume that $\xi<\Om$ and recall that $S_\mathbb X(\chi)$ is the least element of $\mathbb X  $ above $\chi$.
It is easy to check that $\gt(0) =S_\mathbb X (0)$ is the least element of $\mathbb X$ (recall that $0\notin\mathbb X$ by assumption), hence $\gt(0)\notin\mathbb X'$, while $\gt(\xi+1) = S_\mathbb X(\gt(\xi))$, so that also $\gt(\xi+1)\notin\mathbb X'$.
Finally, if $\xi\in {\rm FIX}(\mathbb X)$, then $\xi=\gt(\gamma)$ for some $\gamma>\xi$.
Let $\theta = \sx{\mathbb X}\xi$; we claim that $\theta=\gt(\xi)$.
If $\zeta<\xi$, then since $\zeta$ is countable, $\zeta=\zeta^*<\xi=\gt(\gamma)$, and moreover $\zeta<\xi<\gamma$ so $\gt(\zeta)<\gt(\gamma)=\xi < \theta $, hence $\theta$ is a $\gt$-candidate for $ \xi $.
Moreover, since $\xi$ is also countable, $\xi=\mc\xi<\gt(\xi)$, so $\theta$ must be the least candidate.
\end{proof}

As we will see, Lemma \ref{lemmNotFix} is sharp in the sense that for any other $\xi<\ve_{\Om+1}$, $\gt(\xi) \in \mathbb X' $.
In order to show this, we will find sequences in $\mathbb X$ converging to all such $\xi$.
We will use the function $\tau (\xi)$ defined in Section \ref{secOrd}.

\begin{lemma}\label{lemmContinuousCase}
Fix a club $\mathbb X\subseteq(0,\Om)$ and let $\al=\Om\tilde\al$ and $0<\beta<\Om$.
If $\al+\be \in \constr\cap {\rm Lim} \setminus {\rm FIX}$ and $\beta_n\nearrow\beta$, then $\gt(\al+\beta_n )\nearrow \gt(\al+\beta)$.
\end{lemma}

\begin{proof}
Let $\theta=\sup_{n<\om} \gt(\al+\be_n)$.
First we observe that
\[ \al+\be_n < \al+\be_{n+1} <\al+\be\]
and
\[(\al+\be_n)^*\leq (\al+\be_{n+1})^* \leq (\al+\be )^* <\gt(\al+\be),\]
from which it follows that $\big ( \gt(\al+\be_n)\big) \nearrow \theta$ and $\theta\leq \gt(\al+\be )$.

In order to show that $\theta \geq  \gt(\al+\be )$, it suffices to show that $\theta$ is a $\gt$-candidate for $\al+\be$.
Suppose that $\xi<\al+\be$ and $\xi^*<\theta$.
Then, $\xi<\al+\be_n$ for some $n$ and $\xi^* < \gt(\al+\be_m ) $ for some $m$.
It follows that $\gt(\xi) <\gt(\al+\be_{\max\{n,m\}})<\theta$.

It remains to show that $\theta>(\al+\be)^*$.
We have that $\theta>(\al+\be_0)^* \geq \al^*$, so we must show that $\theta>\be$.
To this end, we consider the following cases.
\begin{Cases}
\item ($\be\notin\mathbb X$). If $\be<\min\mathbb X$, then $\beta <\gt(0) \leq \gt(\al+\be_0)<\theta$, so we may assume that $\be\geq \min\mathbb X$.
Let $\zeta=\sup (\mathbb X\cap \be)$; note that $\zeta\in\mathbb X$ since $\mathbb X $ is closed.
Since $\zeta<\be$, $\zeta< \be_n $ for some $n$.
But then $\be_n< \gt(\al+\be_n)\in \mathbb X$, so by the way we defined $\zeta$, $\be <\gt(\al+\be_n)$ as well.

\item ($\be\in\mathbb X$).
By Corollary \ref{corOmTower}, we can write $\be=\gt(\ga)$.
Since $\al+\be\notin{\rm FIX}$, either $\be\leq \al^*$ or $\ga < \al+\be$.
If $\be\leq \al^*$ then $\be<\gt(\al+\be_0)<\theta$.
Otherwise, $\ga < \al+\be$ and hence $\ga<\al+\be_n$ for some $n$.
Since $ \gt(\ga) = \be$, we have that $\ga^*<\be$, hence $\ga^*< \be_m $ for some $m$.
It follows that $ \be = \gt(\ga) < \gt(\al+\be_{\max\{m,n\}})<\theta $, as needed.
\qedhere
\end{Cases} 
\end{proof}

Before showing how to approximate the $\gt$ functions in other cases, we need to define an auxiliary value.
Basically, $\gt^* (\xi)$ will be a countable ordinal which is added to fundamental sequences in order to ensure that the maximal coefficients are large enough for convergence.

\begin{definition}\label{defStar}
For $\xi<\ve_{\Om+1}$, we set
\[
\gt^* (\xi) =
\begin{cases}
\gt(\zeta)&\text{if $\xi =\zeta+1$}\\
 \tau(\xi)&\text{if $\xi\in {   \rm FIX}(\mathbb X)$}\\
 0&\text{otherwise.}\\

\end{cases} 
 \]
\end{definition}

This operation can be used to bound certain values of the $\gt$ function.
First we note that $\gt^*(\xi)$ is always below $\gt(\xi)$.

\begin{lemma}\label{lemStarBelowTheta}
If $\xi<\ve_{\Om+1}$ then $\gt^*(\xi)<\gt(\xi)$.
Moreover if $\gt^*(\xi)>0$ then $\xi^* \leq \gt^*(\xi) $.
\end{lemma}

\begin{proof}
The claims are trivial if $\gt^*(\xi)=0$. Otherwise we either have that $\gt^*(\xi)=\xi^*<\gt(\xi)$, or else $\xi$ is a successor and $\gt^*(\xi)=\gt(\xi-1)<\gt(\xi)$, while we have $\xi^*\leq (\xi-1)^*+1 \leq \gt(\xi-1)$.
\end{proof}

In contrast, $\gt^*(\al+\be)$ is an {\em upper} bound for $\gt (\al+\be')$ when $\be'<\be$ and $\be$ is countable.

\begin{lemma}\label{lemThetBetPri}
Let $\mathbb X\subseteq(0,\Om)$ be club and $\gt=\gt_\mathbb X$.
Let $\xi =\al+\be$ with $\al=\Om\tilde \al$ and $\be<\Om$, and let $\be'<\be$.
If $\gt^*(\xi)>0$, then $\gt(\al+\be')\leq \gt^*(\xi)$.
\end{lemma}

\begin{proof}
If $\ga=\al+\be'$ with $\be'<\be$, then clearly $\be>0$.
If $\be$ is a successor then $\gt^* (\xi) = \gt(\xi-1 ) \geq \gt(\al+\be' ) $.
If $\be$ is a limit then $\xi\in {\rm FIX}( \mathbb X )$ implies that $\al^*<\be$ and $\be =\gt(\gamma')$ for some $\gamma'> \xi$.
But then $\al+\be'<\ga'$ and  $ (\al+\be')^* = \max \{\al^*,\be'\} < \be = \gt(\ga')$, so that $\gt(\al+\be') < \gt(\ga') = \be =\gt^* ( \xi )$.
\end{proof}

Thus when $\gt^*(\al+\be)>0$ we have that, for $\be'<\be$, $\gt(\al+\be')\leq \gt^*(\al+\be)<\gt (\al+\be)$, so the value $\gt (\al+\be)$ cannot be approximated by simply approximating $\beta$.
In these cases, we instead approximate $\al$; but we need to add a term $\gt^*(\al+\be)$ to ensure that the maximal coefficients are large enough.
The following two lemmas make this precise.
Note that Lemma~\ref{lemmThetaTau} uses an approximation to $\al$ of the form \eqref{eqTauConv} we have discussed.

\begin{lemma}\label{lemmThetaTau}
Let $\mathbb X\subseteq(0,1)$ be club, $\xi=\al+\be<\ve_{\Om+1}$ with $\al=\Om\tilde\al > 0 $ and $\beta<\Om$, and $\tau=\tau(\al) < \Om$.
If $\al+\beta \in   {\rm JUMP}(\mathbb X)$ and $\tau_n\nearrow \tau$, then $\gt \big (\fs\al{\tau_n}+ \gt^* (\xi) \big ) \nearrow \gt(\al+\beta)$.
\end{lemma}

\begin{proof}
It is readily checked using Lemmas~\ref{lemmBoundMC}(\ref{itBoundMCConverge}) and \ref{lemmOmAlph}, along with the fact that $\tau$ is a limit since $\al$ is, that $\fs\al{\tau_n}+ \gt^* (\xi) < \fs\al{\tau_{n+1}}+\gt^* (\xi) <\al+\be$ and
\[(\fs\al{\tau_n}+ \gt^* (\xi) )^* \leq ( \fs\al{\tau_{n+1}}+ \gt^* (\xi))^* \leq \max\{  \al ^* ,\gt^* (\xi)\} < \gt (\al+\be) ,\]
so $ \big ( \gt (\fs\al{\tau_n}+\gt^* (\xi))  \big)_{n<\om} $ is an increasing sequence below $\gt(\al+\be)$.

Now, let $\theta:=\sup_{n<\om} \gt(\fs\al{\tau_n}+ \gt^* (\xi)  )$.
In order to prove that $\theta\geq \gt(\al+\be)$, it suffices to show that $\theta$ is a $\gt$-candidate for $\al+\be$.
Since $\al+\beta \in  {\rm JUMP}(\mathbb X)\setminus\{0\}$ we have that $\gt^*(\xi)>0$, hence Lemma~\ref{lemStarBelowTheta} yields $\xi^* \leq \gt^*(\xi)$, and thus $\xi^*<\gt (\fs\al{\tau_0}+\gt^* (\xi))<\theta$.

Finally, let $\ga<\al+\be$ be such that $\ga^*<\theta$.
We must find $n$ such that $\gt(\ga) < \gt (\fs\al{\tau_n}+ \gt^* (\xi))$.
Consider two cases.

First assume that $\ga<\al$.
Since $\ga^*< \theta  $, we can find $n $ so that $\ga^* < \gt \big  ( \fs\al{\tau_{n }} + \gt^* (\xi) \big )$, and since $\tau_n\to\tau$, we may assume that $\ga< \fs\al{\tau_{n }}  $, so that $\gt(\ga)<\gt(\fs\al{\tau_n} + \gt^* (\xi) )$, as needed.

Otherwise, $\ga=\al+\be'$ with $\be'<\be$.
By Lemma \ref{lemThetBetPri}, $\gt(\ga) \leq \gt^*(\xi)<\gt(\fs\al{\tau_0} + \gt^* (\xi) ) <\theta$.
\end{proof}

\begin{lemma}\label{lemmThetaOm}
Let $0<\al=\Om\tilde\al <\ve_{\Om+1} $ and $\beta<\Om$ be such that $\al+\be \in   {\rm JUMP}(\mathbb X)$, and assume that $\tau(\al) = \Om$.

Define $\theta_0 = \gt^* (\al+\be)$ and $\theta_{n+1} = \gt \big  ( \fs\al{\theta_n}\big )$.
Then, $\theta_n \nearrow \gt(\al+\beta)$.
\end{lemma}

\begin{proof}
We have that $\theta_n<\theta_{n+1}$ since $  \fs\al{\theta_n} ^*\geq \theta_n$, so $\theta_{n+1} =\gt \big (\fs\al{\theta_n}  \big )  > \theta_n$.
By induction on $n$ we can assume that $\theta_n< \gt(\al+\be)$, so that $\fs\al{\theta_n}  <\al+\be$ and
\[ \fs\al{\theta_n}  ^*  \leq \max\{  \al ^* ,\theta_n \} < \gt (\al+\be) \]
imply that $\theta_{n+1}< \gt(\al+\be) $.
Thus, $ (\theta_n)_{n<\om} $ is an increasing sequence below $\gt(\al+\be)$.

Now, let $\theta:=\sup_{n<\om} \theta_n$.
In order to prove that $\theta\geq \gt(\al+\be)$, it suffices to show that $\theta$ is a $\gt$-candidate for $\al+\be$.
First note that if $\gt^*(\al+\be)>0$ then Lemma~\ref{lemStarBelowTheta} yields $(\al+\be)^* \leq \gt^*(\al+\be) =\theta_0<\theta $.
Otherwise, $\gt^*(\al+\be)=0$, so that $\al+\be\in\Om\sf Ord$ and hence $\beta=0$.
By Lemma~\ref{lemmStarPlusOne}, $\fs \al 1^*+1\geq   \al^*$.
Note that $\theta_1=\gt(\fs\al 0)>0$, so $\theta_2=\gt(\fs\al {\theta_1}  )\geq \gt(\fs \al 1)>\fs \al 1^*$.
It follows that $\theta_3>\fs \al 1^*+1\geq \al^*$.
In either case, $(\al+\be)^*<\theta$.

Let $\ga<\al+\be$ be such that $\ga^*<\theta$.
We must find $n$ such that $\gt(\ga) < \theta_n$.
If $\ga<\al$, we use the fact that $\ga^*< \gt(\ga)  $ and induction to find $n $ so that $\ga^* < \theta_n$.
It follows by Lemma \ref{lemmBoundMC}(\ref{itBoundMCTwo},\ref{itFundPropMajor}) that $\ga<\fs\al{\theta_n}$, and we conclude that $\gt(\ga)<\theta_{n+1}<\theta$.

Otherwise, $\ga=\al+\be'$ with $\be'<\be$.
By Lemma \ref{lemThetBetPri}, we conclude that $\gt(\ga) \leq \gt^*(\al+\be)= \theta_0 <\theta$.
\end{proof}

With this, it follows that Lemma \ref{lemmNotFix} is sharp.

\begin{proposition}\label{propNotFix}
If $\xi \in \constr$, then $\gt(\xi)\notin \mathbb X'$ if and only if $\xi<\Om$ and $\xi\in {\rm JUMP}(\mathbb X) $.
\end{proposition}

\section{Regular Bachmann systems}\label{secRegularBS}

The results of the previous section will guide us in defining fundamental sequences for the functions $\gt_\mathbb X$.
Intuitively, a notation system for an ordinal $\Lambda$ is a collection of functions $\mathcal F$ on the ordinals (including at least one constant) such that every $\lambda<\Lambda$ is the value of some term $t$ (see Section~\ref{secNorms}).
Having represented $\lambda $ in this form, we can assign to it a norm, $\|\lambda\|$, measuring the size of this term $t$.
Assuming that $\mathcal F$ is finite, there will be finitely many ordinals of any given norm.
The latter gives rise to the abstract definition of a norm.

\begin{definition}
If $\Lambda$ is any ordinal, a function $\|\cdot\|\colon \Lambda\to\mathbb N$ is a {\em norm} if whenever $n\in\mathbb N$, the set
\[\{\al<\Lambda: \|\al\| \leq n\}\]
is finite.
\end{definition}

Given a countable ordinal $\Lambda$, a system of fundamental sequences on $\Lambda$ is an assignment to each limit $\lambda<\Lambda $ of a sequence $(\cfs\lambda n)_{n<\om}$ converging to $\lambda$.
While this general definition is sound, it is convenient to require additional conditions on $\cfs\cdot\cdot $ to ensure uniform computational behavior of the fast-growing hierarchy based on $\Lambda$ (see Section \ref{sectHardy}).
Let us make these conditions precise.

\begin{definition}[\cite{BCW}]
Let $\Lambda$ be a countable ordinal and $\cfs \cdot \cdot \colon \Lambda\times \mathbb N\to\Lambda$.
\begin{enumerate}[label=(B\arabic*)]

\item We say that $ (\Lambda,\cfs \cdot \cdot) $ is a {\em system of fundamental sequences} on $\Lambda$ if for all $\al<\Lambda$ and $n\in\mathbb N$,
\begin{enumerate}

\item $\cfs 0n = 0$,

\item $\cfs{\al}n = \be$ if $\al=\be+1$, and

\item if $\al \in \rm Lim$, then $ (\Lambda,\cfs \cdot \cdot) $ {\em converges on $\alpha$,} i.e.~$ \cfs\al n \nearrow \al$ as $n\nearrow \infty$.

\end{enumerate}

\item\label{itBach} A system of fundamental sequences $ (\Lambda,\cfs \cdot \cdot) $ has the {\em Bachmann property} if whenever $0<x<\om$ and $\cfs\la x< \eta< \la$, it follows that $\cfs\la x \leq \cfs \eta 1$.
A system of fundamental sequences with the Bachmann property is a {\em Bachmann system.}





\item If $(\Lambda,\cfs\cdot\cdot)$ is a Bachmann system and $\|\cdot\|$ is a norm on $\Lambda$, we say that $(\Lambda,\cfs\cdot\cdot,\|\cdot\|)$ is a {\em regular Bachmann system} if whenever $\eta<\la$, it follows that $\eta\leq \cfs\la{\|\eta\|}$.

\item The regular Bachmann system $(\Lambda,\cfs\cdot\cdot,\|\cdot\|)$ is {\em Cantorian} if it moreover satisfies:
\begin{enumerate}

\item If $\al+\be$ is in Cantor normal form and $\be>0$ then $\cfs{(\al+\be)}n =\al+\cfs\be n $.

\item If $m,n<\om$ then $\om^m\cdot n = \cfs {\om^{m+1}} n$.

\item There is a primitive recursive function $h$ such that:
\begin{enumerate}

\item Whenever $\al_1\geq\ldots\geq \al_n$,
\begin{align*}
\max\{n,&\|\al_1\|,\ldots, \|\al_n\|\}  \leq  \| \om^{\al_1} +\ldots +\om^{\al_n}\|\\
&  \leq h\big ( \max\{n,\|\al_1\|,\ldots, \|\al_n\|\} \big ).
\end{align*}

\item For all $m<\om, m < \|\om^m\|\leq h(m)$.

\end{enumerate}

\end{enumerate}

\end{enumerate}

\end{definition}

\begin{remark}
Note that in \cite{BCW}, condition \ref{itBach} is stated with $\cfs\eta 0$ in place of $\cfs\eta 1$.
This essentially corresponds to a change of variables $ x\mapsto x+1$ and does not affect the results we cite or their proofs.
A similar remark applies to other properties on the list.
\end{remark}

Our main goal in the sequel is to prove that Buchholz's fundamental sequences provide a regular Bachmann system.
Then, in Section \ref{sectHardy}, we explore the applications of these properties to provably total computable functions.

\section{Buchholz systems of fundamental sequences}\label{secFS}

The results of Section \ref{secXLim} can be used to design fundamental sequences for notation systems based on theta functions.
These are due to Buchholz, although our notation is a bit different.
For $\xi=\al+\be$ with $\al=\Om\tilde\al$ and $\be<\Om$, we define
\[
\check \xi=
\begin{cases}
 \al&\text{if $\gt^*(\xi) >0$}\\
\xi&\text{otherwise.}
\end{cases}
\]
The intuition is that $\check\xi$ is the `part' of $\xi$ that we need to apply fundamental sequences to in order to approximate the value of $\gt(\xi)$; as per the discussion in Section~\ref{secXLim}, $\check \xi$ should be $\xi$ itself when $\xi$ is a point of continuity, but will be an initial part (the largest multiple of $\Om$ below $\xi$) when $\xi$ belongs to ${\rm JUMP}(\mathbb X)$.
It is also worth noting that $ \check \xi $ is rarely zero.

\begin{lemma}\label{lemmTauZero}
If $\xi\in\constr$ then $0= \check \xi <\xi$ iff $\gt(\xi)\notin\mathbb X'$.
\end{lemma}

\begin{proof}
We note that $\check \xi=0$ if and only if $\xi=0$ or $\xi$ is countable and $\check \xi<\xi$, so that $\xi \in {\rm JUMP}(\mathbb X)\cap \Om$; by Proposition~\ref{propNotFix}, this occurs exactly when $\gt(\xi)\notin\mathbb X'$.
\end{proof}

We also define $\gt^{(i)}(\xi)$ recursively by setting
\begin{enumerate}

\item $\gt^{(0)}(\xi) = \gt^*(\xi)$ and

\item $\gt^{(i+1)}(\xi) = \gt\big ( \fs\al{\gt^{(i)}(\xi)} \big )$.

\end{enumerate}

With this, we are ready to define Buchholz systems of fundamental sequences.
The general idea is that these systems behave in a specified way on elements of $\mathbb X'$, but otherwise may be defined freely in a case-by-case basis.
In particular, if $\xi\in \Om\cap {\rm JUMP}(\mathbb X) $, then as noted above, $\gt(\xi)\notin \mathbb X'$, and hence Definition~\ref{defBuchSys} does not prescribe any value for $\cfs{\gt(\xi)}n$, aside from the assumption that fundamental sequences are increasing and lie below the original ordinal.

\begin{definition}\label{defBuchSys}
Let $\mathbb X\subseteq (0,\Om)$ be club and $\gt=\gt_\mathbb X$, and let us define $\Lambda= \{\xi\in\constr: \tau(\xi)<\Om\} $.
Then, an assignment $\cfs\cdot\cdot\colon \Lambda \times\mathbb N \to \constr$ is a {\em Buchholz system} (for $\mathbb X$) if for all $\xi\in\Lambda$ and $n\in\mathbb N$:
\begin{enumerate}

\item $\cfs 0n = 0$ and if $\xi$ is a successor then $\cfs\xi n = \xi-1$.

\item If $\xi$ is a limit then $\cfs\xi n<\cfs\xi{n+1}<\xi$.

\item

If $\xi  \geq \Om$ and $\tau=\tau(\xi)\in \rm Lim$ then $\cfs\xi n = \fs\xi{\cfs \tau n}$.

\item If $0<\tau(\check \xi)<\Om$ then
$\cfs{\gt(\xi)}n  = \gt \big ( \cfs{\check \xi} n +\gt^*(\xi) \big )$.

\item If $\tau(\check \xi)=\Om$, then
$
\cfs{\gt(\xi)}n=
\gt^ {(n )} (\xi).
$

\end{enumerate}
\end{definition}

\begin{lemma}
Any Buchholz system converges on every element of $\mathbb X'$, provided it converges elsewhere.
\end{lemma}

\begin{proof}
That fundamental sequences are monotone and convergent follows inductively using Lemmas \ref{lemmContinuousCase}, \ref{lemmThetaTau}, and \ref{lemmThetaOm}.
\end{proof}

It will be useful to state a general characterization of the values of the fundamental sequences for elements of $  \mathbb X'$.

\begin{lemma}\label{lemmBackwards}
Let $\cfs\cdot\cdot$ be any Buchholz system and let $\xi\in\constr$ and $\gt(\xi)\in\mathbb X'$.
\begin{enumerate}
\item If either $n>0$ or $\tau(\xi)<\Om$, there exists $\tilde\xi<\xi$ such that $\gt(\tilde\xi) = \cfs{\gt(\xi)}n$ and either
\begin{enumerate}
\item \label{itBackOne} $\mc{\tilde \xi}\geq\mc\xi$, or

\item \label{itBackTwo}$\mc{\tilde \xi}\geq \cfs{\mc\xi}n$ and  $\mc\xi \neq \gt(\zeta)$ for any $\zeta>\xi$.

\end{enumerate}

\item If $\tau(\xi) =\Om$, then either $\cfs{\gt(\xi)} 0 = 0 $ or $\cfs{\gt(\xi)} 0 =\gt(\tilde\xi)$ for some $\tilde\xi$ with $\tilde\xi+1 \geq \xi$.

\end{enumerate}
\end{lemma}

\begin{proof}
For the first claim, note that by Lemma~\ref{lemmTauZero} and the assumption that $\gt(\xi)\in\mathbb X'$ we obtain $\check \xi>0$. 
Inspecting Definition~\ref{defBuchSys}, we see that $\cfs{\gt(\xi)}n = \gt(\tilde\xi)$ with $\tilde \xi=\fs{\check\xi}\delta+\theta$ for some $\delta,\theta$ with $\delta<\tau(\check\xi)$.
It follows that $\tilde \xi<\xi$.
To continue, we must consider several possibilities.
\begin{Cases}

\item ($0< \gt^*(\xi)  $).
Since $n>0$, inspection of Definition~\ref{defBuchSys} shows that $\gt^*(\xi)\leq \max \{\delta,\theta\}$.
Since $\xi^*\leq \gt^*(\xi)$, in this case we also obtain \eqref{itBackOne}.

\item ($\xi^*=\fs\xi 1^*$ and $\gt^*(\xi) =0 $). Then, $\check \xi=\xi$, $\theta=0$ and since $n>0$, $\delta\geq 1$.
Thus $\xi^*=\fs {\xi} 1^*\leq \fs {\check \xi} \delta^*$, and \eqref{itBackOne} holds.

\item ($\fs\xi 1^*<\xi^*$ and $\gt^*(\xi) =0$).
Since $\gt^*(\xi) =0$, $\check \xi=\xi$.
Let $\tau=\tau(\xi)$.
In view of Lemma~\ref{lemmStarPlusOne}, there are two cases to consider.
\begin{Cases}
\item ($\xi^*=\tau \in \rm Lim$). In this case, $\tilde\xi=\fs\xi{\cfs\tau n}$.
From $\gt^*(\xi) =0$ we see that $\tau$ is not of the form $\gt(\zeta)$ with $\zeta>\xi$ (otherwise, $\gt^*(\xi)=\tau$).
Moreover, by Lemma~\ref{lemmBoundMC}(\ref{itBoundMCTwo}), $ \tilde \xi^*\geq \cfs\tau n=\cfs{\xi^*} n$, so \eqref{itBackTwo} holds.

\item ($\xi^*=\fs \xi 1^*+1 $).
In this case, we immediately have that $ \tilde \xi^* \geq \fs \xi 1^* =\cfs{\xi^*} n$.
Note that $\xi$ must be uncountable, since otherwise $\xi=\xi^*$ and by Proposition~\ref{propNotFix}, $\gt(\xi)\notin\mathbb X'$.
If $\zeta>\xi$, it follows that $\gt(\zeta)$ is a limit (as $\zeta$ is uncountable) and thus $\xi^* \neq \gt(\zeta)$.
Thus, \eqref{itBackTwo} holds in this case as well.
\end{Cases}
\end{Cases}

For the second claim, we have defined $\cfs{\gt(\xi)}0=\gt^*(\xi)$ in this case, and it can be checked by the definition that $\gt^*(\xi)$ is of the required form.
\end{proof}

Below, recall that we defined $\mathbb P$ to be the set of principal numbers, $\vartheta=\gt_\mathbb P$, and $\nt=\gt_{1+\sf Ord}$.
In order to define Buchholz systems, it suffices to define $\cfs\xi n$ in cases where $\xi\in \Om\setminus (\{0\}\cup {\rm Succ}\cup \mathbb X')$, as other cases are covered by Definition~\ref{defBuchSys}.
In the case of $\mathbb X=1+\sf Ord$, this case is in fact vacuous.

\begin{proposition}
There is a unique Buccholz system for $1+\sf Ord$, which we denote by $\cfs\cdot\cdot_\sigma$.
\end{proposition}

\begin{proof}
Every ordinal is either zero, a successor, or a limit of $1+\sf Ord$, so Definition~\ref{defBuchSys} already prescribes values for $\cfs\cdot\cdot_\sigma$ in all cases.
\end{proof}

However, in a more general setting, one needs to explicitly provide values for $\cfs \xi n$ when this is not covered by Definition~\ref{defBuchSys}.
In the case of $\mathbb P$, this includes all additively decomposable limits, as well as all values of $\gt(\xi)$ with uncountable $\xi$ (which, as we will see in Lemma~\ref{lemEpsilon}, are precisely the $\ve$-numbers).

\begin{definition}
Given $\xi<\vt(\ve_{\Om+1})$ and $n<\om$, we define $\cfs\xi n_\vt$ to be the unique Buchholz system for $\mathbb P$ such that
\begin{enumerate}

\item 
$\cfs{ (\al+\be) }n_\vartheta = \al+ \cfs \be n$ if $\be>0$ and $ \al+\be$ is in Cantor normal form, and

\item  $\cfs{\vartheta (\be)} n_\vartheta = \gt ^*(\be)\cdot n$ if $\be\in \Om \cap   {\rm JUMP}(\mathbb X)$.

\end{enumerate}
\end{definition}

Note that this definition should be understood recursively, so that in the first clause one must compute $\cfs \be n$ according to the above clauses or Definition~\ref{defBuchSys}, as appropriate.
We may omit the subindex $\vt$ or $\nt$ when clear from context.
It is readily checked that these systems of fundamental sequences converge everywhere, yielding the following.

\begin{theorem}\label{theoIsFun}
The assignments $\cfs\cdot\cdot_\vt,\cfs\cdot\cdot_\nt \colon \Lambda\times  \mathbb N \to \Lambda$ are systems of fundamental sequences.
\end{theorem}

We have seen that $\nt$ is the successor function when restricted to countable ordinals.
Similarly, $\vartheta$ is essentially $\omega$-exponentiation, albeit `skipping' jumps.
The following is proven by an easy induction based on the fundamental sequences for $\vartheta$.

\begin{lemma}\label{lemmThetaExp}
Let $\xi<\Om$ and write $\xi=\la+n$ where $\la$ is a limit or zero and $n$ is finite.
Then,
\[
\vartheta(\xi)=
\begin{cases}
\om^\xi&\text{if $\lambda\notin{\rm JUMP}(\mathbb P)$,}\\
\om^{\xi+1}&\text{if $\lambda\in{\rm JUMP}(\mathbb P)$.}
\end{cases}
\]
\end{lemma}

With this in mind, the fundamental sequences with respect to $\vartheta$ behave as expected with respect to Cantor normal forms.
The proof of the following is derived easily from the definitions.

\begin{lemma}\label{lemCantorOne}\
\begin{enumerate}

\item If $\al+\be <\vartheta(\ve_{\Om+1})$ is in Cantor normal form, then
\[ \cfs{(\al+\be)}n _\vartheta = \al+\cfs\be n_\vartheta .\]

\item If $m,n <\om$, then $ \cfs{\om^{m+1}}n_\vartheta = \om^m\cdot n $.
\end{enumerate}

\end{lemma}

Note that the above fails for $\nt$.
Thus the fundamental sequences based on $\vartheta$ extend the familiar ones for $\varepsilon_0$. The latter have the property that they almost always take infinite values in the sense that if $0<\cfs \alpha n<\om $ and $\alpha$ is infinite, then $\alpha =\om$.
The reader should be warned that this is not necessarily the case anymore, since for example $\cfs{\vartheta(\Om)}1 _\vartheta = \vartheta(0) = 1$.
However, finite values of the fundamental sequence are still rather rare.

\begin{lemma}\label{lemmFinVal}
If $\om\leq \alpha<\vartheta(\ve_{\Om+1})$ and $m$ are such that $1<\cfs \al m_\vartheta <\om $, then $\al=\om$.
\end{lemma}

\begin{proof}
By Lemma~\ref{lemCantorOne} we see that if $\om\leq \alpha=\ga+\de$ in Cantor normal with $\de>0$, then $\cfs \al m_\vartheta = \ga+\cfs \de m_\vartheta\geq \om $.
In any case that $\cfs \al m_\vartheta = \vartheta(\al')$, the only possible finite value is $1$, and if $\cfs \al m_\vartheta = \vartheta^*(\al)$, then either $\vartheta^*(\al) = 0$ or it is of the previous form.
So we are left with the case where $\al = \vartheta(\xi)$ with $\xi<\Om$ and $\xi\in {\rm JUMP}(\mathbb P)$, so that $\cfs \al m_\vartheta = \vartheta^*(\xi) m$.
Under the assumption, we must have $\vartheta^*(\xi) = 1$, which is only possible if $\xi=1$, i.e.~$\al=\om$.
\end{proof}

Finally we note that uncountable ordinals produce $\ve$-numbers under $\vt$.

\begin{lemma}\label{lemEpsilon}
If $\xi\in\constr$ and $\xi\geq\Om$ then $\vartheta(\xi)=\om^{\vartheta(\xi)}$.
\end{lemma}

\begin{proof}
Let $\theta=\vartheta(\xi)$.
Note that in view of Proposition~\ref{propNotFix}, we have that $\xi\in\mathbb X'$ and moreover by Lemma~\ref{lemmBackwards}, for all $n>0$ there is $ \xi_n<\xi$ such that $\cfs\theta n=\vartheta(\xi_n)$.
If $ \xi_n \geq \Om $ infinitely often, then we can use the induction hypothesis to see that $ \vartheta(\xi_n)$ is a $\ve$-number, hence $\vartheta(\xi)$ is a limit of $\ve$-numbers and also a $\ve$-number.
Otherwise, we may write $\check \xi=\Om^\delta\rho+\eta'$ in $\Om$-normal form and observe that the only way to have $\fs{\check \xi}\lambda<\Om$ for any $\lambda>0$ is to have $\eta'=0$ and $\delta=\rho=1$, which may only occur when $\xi=\Om+\eta$ for some $\eta<\Om$ and $\xi\in{\rm JUMP}(\mathbb P)$.
We then see by Lemma~\ref{lemmThetaExp} and an easy induction that that $\fs{\vartheta(\Om)}{n}=\om_n({\vartheta^*(\xi)+x})$ (where $x\in\{0,1\}$), so $\vartheta(\Om)$ is the smallest $\ve$-number above $\vartheta^*(\xi)$.
\end{proof}

Note that $\ve$-numbers are closed under addition, multiplication, and $\om$-expo\-nen\-tiation, a fact that will be useful later.
Next we note that Buchholz systems are designed to `skip over' values of $\gt$ with large arguments, in the following sense.

\begin{lemma}\label{lemmBackwardsExt}
Let $\mathbb X\in \{1+{\sf Ord},\mathbb P\}$, $\al<\gt(\ve_{\Om+1})$, and $n>0$.
Suppose that $\zeta $ is such that $\gt(\zeta)<\al$ and for all $\xi> \zeta$, $\al\neq\gt(\xi)$.
Then, $\gt(\zeta)\leq \cfs\al n$, and the inequality is strict if $\al\in\mathbb X'$.

If $\al = \gt(\xi)$ for some $\xi$ with $\tau(\check \xi)=\Om$ and $\gt^*(\xi)>0$, then also $\gt(\zeta)\leq \cfs\al 0$.
\end{lemma}

\begin{proof}
Note that $\cfs\al n>0$ implies $\al>0$.
With this in mind, proceed by induction on $\al$ and consider the following cases.
\begin{Cases}
\item ($\al\notin \mathbb X$).
Then $\mathbb X=\mathbb P$ and we can write $\al=\om^\ga+\de$ in Cantor normal form, so that $\cfs\al n = \om^\ga+\cfs\de n$.
Since $\gt(\zeta)$ is additively indecomposable, $\gt(\zeta)<\al$ yields $\gt(\zeta)\leq \om^\ga\leq \cfs\al n $.

\item ($\al\in\mathbb X\setminus \mathbb X'$).
For $\mathbb X=1+\sf Ord$, this implies that $\al$ is a successor, hence $\al=\gt(\beta)$ for some $\beta<\Om$ and $\cfs\al n = \beta\geq \gt(\zeta)$.
Otherwise, $\mathbb X=\mathbb P$ and $\al = \gt(\xi)$ for some $\xi$.
It follows that $\cfs\al n = \gt^*(\xi)\cdot n$.
From $\gt(\zeta)<\al$ we see that $\gt^*(\xi)\cdot m >\gt(\zeta) $ for some $m$, but since $\gt(\zeta)$ is additively indecomposable, $\gt(\zeta) \leq \gt^*(\xi) = \cfs{\al}1\leq \cfs\al n$.

\item ($\al\in\mathbb X'$). Write $\al=\gt(\xi)$.
If $ \zeta>\xi$ and $\gt(\zeta)<\gt(\xi) $, then also $\gt(\zeta)\leq\xi^*$ by Lemma \ref{lemmThetaOrd}.
Let $\tilde \xi$ be so that $\cfs{\gt(\xi)}n = \gt(\tilde \xi ) $ as in Lemma \ref{lemmBackwards}.
If $\tilde\xi^* \geq\xi^*$ then $\gt(\zeta) \leq \tilde\xi^*< \gt(\tilde \xi ) =  \cfs{\gt(\xi)}n$.
Otherwise, case \eqref{itBackTwo} of Lemma \ref{lemmBackwards} holds, so that $\tilde\xi^* \geq \cfs{\xi^*}n $ and $ \xi^* \neq \gt(\zeta)$, hence $ \xi^* > \gt(\zeta)$.
The induction hypothesis yields $\cfs{\xi^*}n\geq \gt(\zeta)$, hence $\gt(\tilde\xi) > \tilde\xi^* \geq \cfs{\xi^*}n \geq \gt(\zeta)$.
\end{Cases}

Finally, if $\al=\gt(\xi)$ with $\gt^*(\xi)>0$ and $\tau(\check\xi)=\Om$, Lemma \ref{lemmBackwards} tells us that $\gt^*(\xi) = \gt(\tilde\xi)$ with $\tilde\xi+1\geq \xi$.
Since $\zeta>\xi$ but $ \gt(\zeta) <\gt(\xi)$, we have $\gt(\zeta) \leq \xi^*$.
If $\gt^*(\xi)=\xi^*$, this yields $\gt(\zeta) \leq \gt^*(\xi) = \cfs \al 0$. Otherwise, $\tilde\xi+1 =\xi$.
In this case, $\tilde\xi  =\fs\xi 1$, so Lemma~\ref{lemmStarPlusOne} tells us that either $\tilde\xi^*=\xi^*$ or $\tilde\xi^*+1=\xi^*$ (as $\xi^*$ is not a limit).
In either case we see that $\gt (\zeta) \leq \xi^*\leq  \tilde\xi^*+1$, which, since $\gt (\zeta)$ is a limit, yields $\gt (\zeta) \leq \tilde\xi^*$.
Thus once again, $\gt(\zeta) \leq   \cfs \al 0$.
\end{proof}

\section{Comparing $\vartheta$ and $\nt$}\label{secThetaFun}

The function $\vartheta$ is compared to other systems of notation for $\vartheta(\ve_{\Om+1})$ in \cite{BuchholzSurvey}, but the function $\nt$ is not considered.
Here we give an explicit correspondence between the two.
The two functions coincide from $\Om^2\om$ on, but below this value they can be quite different.
Below $\Om^2$ one normally one has $\nt(\Om(1+\al)+\be) = \om^{\om^\al}\be$, but  when $\Om(1+\al)+\be \in {\rm FIX}(1+{\sf Ord})$ one must make small `corrections' to $\be$ given by the values $x_\al^\be$ and $y_\al$ defined below, both of which are at most one.

\begin{definition}
For $\xi <\ve_{\Om+1} $, we define $\tv\xi$ as follows.
\begin{enumerate}[label=(\alph*)]

\item $\tv 0 = 0$.

\item\label{itOneTheta} If $\al$ is countable, then $\tv{\B(\al)} =    \al+1$.

\item\label{itTwoTheta} Let $\al,\be  $ be countable and $ \be =  \gamma +n$ with $\ga$ a limit or zero.
Define
\begin{multicols}2
\[
x_\al^\be=
\begin{cases}
1 & \text{if $\om^{\om^\al}\ga=\ga$,}\\
0&\text{otherwise.}
\end{cases}
\]

\[
y_\al =
\begin{cases}
1 & \text{if $\om^\al = \al$,}\\
0&\text{otherwise.}
\end{cases}
\]
\end{multicols}
Then,
\[
\tv{\B(\Om (1+\al) +\be)} =
 \om^{\om^{ \al}}(y_\al+   \be+x_\al^\be).
\]

\item
\label{itThreeTheta}
 $\tv{\B(\Om^2+\be)} = \vartheta(\Om+ \be)$.

\end{enumerate}
\end{definition}

The function $\tv\cdot$ gives us a translation between notations based on the $\nt$ function and those based on the $\vartheta$ function.
We begin by showing that $\tv\zeta$ provides an upper bound for $\zeta$.

\begin{lemma}\label{lemmLeqTV}
If $\zeta <\A(\ve_{\Om+1}) $ then $\zeta \leq \tv\zeta$.
\end{lemma}

\begin{proof}
Proceed by induction on $ \zeta$ and consider several cases.
Clearly $0\leq \tv 0$, so we assume that $\zeta = \nt (\xi)$.

\begin{Cases}

\item ($\xi<\Om$).
As shown in Example~\ref{exSigmaSucc}, $\nt$ is just the successor function in this case, yielding the claim.

\item ($\Om\leq\xi<\Om^2$).
Write $\xi=\Om (1+\al) +\be$ and $\be=\ga+n$, where $\ga<\Om$ is a limit or zero, and let $x=x_\al^\be$, $y=y_\al$ be as in the definition of $\tv\xi$.
The claim is that
\begin{align*}
\A(\Om(1+\al)+\be) &\leq \tv{\A(\Om(1+\al)+\be) }\\
 & =  \om^{\om ^\al } (y+ \be + x).
\end{align*}
Let $\theta =  \om^{\om ^\al } (y+ \be + x) $.
We check that $\A(\Om(1+\al)+\be)\leq \theta$.

First we show that $(\Om(1+\al)+\be)^* < \theta$.
It is clear that $\theta>\be$ unless $\om^{\om ^\al }   \be =\be $, but in this case $x=1$ so $\om^{\om ^\al } (y+ \be + x) > \be+1>\be$.
Similarly, $\theta>1+\al$ unless $\om^{\om^\al}=1+\al$.
This is only possible if $\om^\al=\al=1+\al$, whence $\al$ is infinite and $y=1$.
Note that $\be+x\geq 1$ since $\be=0$ implies $\om^{\om^\al}\be=\be$ and thus $x=1$.
Thus $\theta=\om^{\om^\al}(y+\be+x) \geq \al 2>1+\al$.
We conclude that $\theta>(\Om(1+\al)+\be)^*$.

It remains to show that if $\zeta<\xi$ and $\mck\Om\zeta<\theta$, then $\A(\zeta) <\theta$.
If $\zeta<\Om$, then $\zeta=\zeta^*<\theta$, and since $\theta$ is a limit, $\A(\zeta)=\zeta+1<\theta$.
Otherwise, write $\zeta=\Om(1+\al')+\be'$.
For $x' = x_{\al'}^{\be'}$ and $y' = y_{\al'} $, we have by induction that
$\A(\zeta) \leq  \om^{\om ^{\al'} } (y'+ \be' + x')$.
Consider the following cases.
\begin{Cases}
\item ($\be'<\ga$).
Then $y'+\be'+x'<\ga$, since $\ga$ must be a limit.
The induction hypothesis yields
$\A(\zeta) <    \om^{\om ^\al } \ga\leq \theta. $

\item ($\ga\leq \be' <\be$). 
In this case, $x'=x$ and $\beta$ is a successor.
Consider two sub-cases.
\begin{Cases}
\item ($\al'=\al$). Then, also $y'=y$ and the induction hypothesis yields $\A(\zeta) \leq  \om^{\om ^\al } (y+ \be + x - 1)<\theta$.

\item ($\al'<\al$).
The induction hypothesis yields
\begin{align*}
\A(\zeta) & \leq  \om^{\om ^{\al'} } (y'+ \be + x - 1)  = \om^{\om ^{\al'} } y'+ \om^{\om ^{\al'} } (\be + x -1 )\\
& < \om^{\om ^{\al'} } y'+ \om^{\om ^{\al} } (y+ \be + x  ) = \om^{\om ^{\al} } (y+ \be + x  ).
\end{align*}

\end{Cases}

\item ($\be'\geq \be$).
In this case we have that $\be'<\theta$ and $\al'<\al$.
Then, $\be'  =   \om^{\om ^\al } \delta + \eta$ for some $\delta < (y+ \be+x)$ and some $\eta<\om^{\om^\al}$.
The induction hypothesis yields
\begin{align}
\nonumber\A(\zeta)& \leq   \om^{\om ^{\al'} }  \big (
y' + \om^{\om ^\al }\de  + \eta +x'\big) \\
\nonumber &=   \om^{\om ^{\al'} }y'+ \om^{\om ^{\al'} }  \om^{\om ^\al } \de    +  \om^{\om ^{\al'} } (\eta +x')\\
\nonumber& =  \om^{\om ^{\al'} }y' + \om^{\om ^\al } \de  +  \om^{\om ^{\al'} } (\eta +x')\\
\nonumber &<   \om^{\om ^{\al'} }y' +\om^{\om ^\al } \de  +  \om^{\om ^{\al} } = \om^{\om ^{\al'} }y' + \om^{\om ^\al } (\de+1)\\
\nonumber & =  \om^{\om ^\al } (\de+1) \leq \om^{\om ^{\al} }(y+\be+x) = \theta.
\end{align}
\end{Cases}
\item ($\xi\geq \Om^2$).
Write $\xi=\Om^2+\al$.
The claim is that $\nt(\Om^2+\al) \leq \vartheta(\Om+\al) $.
To prove this, we show that $\theta :=\vartheta(\Om+\al) $ is a $\nt$-candidate for $\xi$; since $\Om^2+\al$ is defined as the least such candidate, the inequality follows.

Let us begin by showing that $\xi^*<\theta$.
Write $\al=\Om^\delta\rho+\eta$ in $\Om$-normal form.
Then, on a case-by-case basis we can check that $\xi^*\leq\max\{2,1+\al^*\} $.
For example, if $\delta=2$ and $\rho=\al^*$ we have that $\Om^2+\al=\Om^2(1+\rho)+\eta$ which has maximal coefficient $1+\rho$.
Similarly, $(\Om+\al)^*\geq \max\{1,\al^*\}$, so $\vartheta(\Om+\al) > \max\{1,\al^*\}$.
This implies that $\vartheta(\Om+\al)\geq \om$, so $\vartheta(\Om+\al) >2$, while from $\vartheta(\Om+\al) >  \al^* $ we readily obtain $\vartheta(\Om+\al) >  1+\al^* $ since $\vartheta(\Om+\al)$ is additively indecomposable.
In any case, we obtain $\vartheta(\Om+\al) >\xi^*$.
 
Next we show that if $\tilde \xi<\xi $ is such that $\tilde \xi^*< \theta$, then $\nt(\tilde \xi)< \theta$.
If $\tilde\xi<\Om$ then $ \nt (\tilde \xi) =\tilde\xi^*+1<\theta$ since $\theta$ is a limit.
If $\Om\leq\tilde\xi<\Om^2$, write $\tilde\xi= \Om (1+\tilde \al)+\tilde \be$.
Then, $\nt(\tilde\xi) \leq \om^{\om^{\tilde\al}} (1+\tilde\be+1) $.
Since $\Om+\al\geq \Om$ it follows from Lemma~\ref{lemEpsilon} that $\theta$ is an $\ve$-number and hence closed under sums, products, and $\om $-powers, so $\om^{\om^{\tilde\al}} (1+\tilde\be+1)<\theta$.
Finally, if $\Om^2 \leq \tilde\xi $ then $\tilde \xi = \Om^2 +\tilde \al$ for some $\tilde \al< \al$.
Since the function $\chi\mapsto \Om+\chi$ is strictly monotone on $\chi$, we have that $\Om   +\tilde \al <\Om   + \al $, and moreover $(\Om   +\tilde \al)^* \leq \theta $ because $\theta$ is a limit, so $\theta> 1 + (\Om^2+\tilde \al)^*  \geq  (\Om +\tilde \al)^*$.
Using the induction hypothesis, we obtain $\nt( \Om^2+\tilde\al )\leq \vartheta(\Om +\tilde\al) <\theta $, as needed.\qedhere
\end{Cases}

\end{proof}

Next we prove that $\tv\zeta$ is a {\em lower} bound for $\zeta$.
Here we need a more detailed statement, which will allow us to compare the corresponding Hardy hierarchies later.

\begin{theorem}\label{theoLeqTV}
If $\zeta <\vartheta(\ve_{\Om+1})$, then
\begin{enumerate}

\item  $\zeta = \tv\zeta$, and

\item if $0<m<\om$, then
$ \cfs {\tv\zeta} m_\vartheta \leq \cfs \zeta{m+1 }_\nt  $. 

\end{enumerate}

\end{theorem}

\begin{proof}
Proceed by induction on $ \zeta$.
The first claim follows from the second, since $\zeta\leq\tv\zeta$ by Lemma \ref{lemmLeqTV}, and moreover the induction hypothesis applied to the second claim yields
\[\zeta =   \lim_{m \to\infty} \cfs\zeta m_\nt \geq \lim_{m\to\infty}  \cfs {\tv\zeta} m_\vartheta = \tv \zeta . \]
Thus we focus on the second claim.

The claim is trivial for $\zeta=0$, so we assume otherwise and write $\zeta =\nt(\xi)$.
Proceed by induction on $\zeta$.
For the induction hypothesis to go through, we also need to establish the side condition that if $\de=\om^\de$ and $\tv\zeta=\ga+\de$ in Cantor normal form, then
\begin{equation}\label{eqSide}
\ga+\om^{\cfs \de m}\leq  \cfs{\zeta} {m+1}_\nt.
\end{equation}
Consider the following cases.

\begin{Cases}

\item ($\xi <\Om$).
Then, $\cfs {\nt(\xi)} m_\nt = \xi $ and also $\cfs {\tv{\nt(\xi)}} m_\vartheta = \cfs{(\xi+1)}m_\vartheta = \xi $ for all $m$.

\item ($\Om\leq\xi<\Om^2$).
Write $\xi= \Om (1+\al) +\be$ and let $x=x_\al^\be$, $y=y_\al$ be as in the definition of $\tv\xi$, so that
$\tv\xi = \om^{\om^\al} (y+\be+x)$.
Consider the following sub-cases.
\begin{Cases}

\item ($y+\be+x =0$).
This case is impossible since $\be=0$ implies that $x=1$.

\item ($y+\be+x \in \rm Lim$).
Then, $\be \in \rm Lim$ and $x = 0$, so $\tv \zeta  = \om^{\om^\al}\be >\be  $.
We claim that $\xi \notin {\rm JUMP}(1+{\sf Ord})$; for otherwise, $\be>\al$ and $\be=\nt(\xi')$ for some $\xi'>\xi$.
From $\om^{\om^\al}\be >\be$ we see that $\om^{\om^\al}\rho >\be$ for some $\rho<\be$, whereas the induction hypothesis and the definition of $\nt$ yield
\[\be=\nt(\xi') \geq \nt(\Om(1+\al) +\rho) \geq \om^{\om^\al}\rho >\be,\]
a contradiction.

Next, consider the following sub-cases.
\begin{Cases}

\item ($\be= \ga+\om^{\de+1}$).
Then,
\begin{align*}\cfs\zeta {m+1}_\nt& = \cfs{\nt(\Om(1+\al)+\ga+\om^{\de+1})}{m+1} _\nt\\
& =  \nt\big (\Om(1+\al)+\cfs{(\ga+\om^{\de+1})} {m+1}_\nt \big ) \\
 & \stackrel{\text{\sc ih}}\geq  \nt\big (\Om(1+\al)+\cfs{(\ga+\om^{\de+1})} {m}_\vartheta \big ) \\
&  = \nt\big (\Om(1+\al)+ \ga+\om^{\de}   m \big )\\
 & = \om^{\om^\al}(y + \ga + \om^{\de }  m+ x'),
\end{align*}
where $x'=x^{ \ga + \om^{\de }\cdot m}_\al$.
Meanwhile,
\begin{align*}
\cfs{\tv \zeta}m_\vartheta & = \cfs{(\om^{\om^\al}\be)}m_\vartheta 
 =  \cfs{\om^{\om^\al}\ga+\om^{\om^\al+{\de+1}}}m_\vartheta\\
& = \om^{\om^\al}\ga+ \om^{\om^\al+\de}  m =\om^{\om^\al} (\ga+ \om^\de  m) .
\end{align*}
Clearly
\[\ga+ \om^\de  m \leq y + \ga + \om^{\de }  m + x' . \]
From here it is easy to see that $ \cfs {\tv\zeta} m_\vartheta \leq \cfs \zeta{m+1 }_\nt$. 

\item ($\be= \ga+\om^\de$ with $\de\in \rm Lim$).
Then,
\begin{align*}\cfs\zeta {m+1}_\nt& = \cfs{\nt(\Om(1+\al)+\ga+\om^\de)}{m+1} _\nt\\
& =  \nt\big (\Om(1+\al)+\cfs{(\ga+\om^\de)} {m+1}_\nt \big ) \\
 &  \stackrel{\text{\sc ih}}\geq \nt  (\Om(1+\al)+ \ga+ \om^{\cfs \de  m_\vartheta}    ) \\
& = \om^{\om^\al}(y + \ga + \om^{\cfs { \de} m_\vartheta}+ x'),
\end{align*}
where $x'=x^{ \ga + \cfs{\om^\de} m_\vartheta}_\al$ and the induction hypothesis also uses \eqref{eqSide} if $\de=\om^\de$.

On the other hand, note that $\om^{\al}+\de>\de$ since otherwise $\om^{\om^\al}\be=\be$, so
\begin{align*}
\cfs{\tv \zeta}m_\vartheta & = \cfs{\big (\om^{\om^\al}(\ga+{\om^\de})\big )}m_\vartheta 
 =  \cfs{\om^{\om^\al}\ga+\om^{\om^\al+\de}}m_\vartheta \\
& = \om^{\om^\al}\ga+ \om^{\om^{\al}+\cfs\de m_\vartheta}
  = \om^{\om^\al}(\ga+ \om^{ \cfs\de m_\vartheta}) .
\end{align*}
From this it is clear that $ \cfs {\tv\zeta} m_\vartheta \leq \cfs \zeta{m+1 }_\nt$.

\end{Cases}

\item ($y+\be+x \in \rm Succ$).
In this case, it will be convenient to first show that

\begin{claim}\label{claimJump}
$ \xi \in {\rm JUMP} (1+{\sf Ord})$ and $\nt ^*(\xi) = \om^{\om^\al}(y+\be+x-1)$.
\end{claim}

\begin{proof}[Proof of Claim.]
In order to prove the claim note that by definition of $x$, $ \be$ is either zero, a successor, or a limit with $\om^{\om^\al}\be = \be$.
Thus we consider each of these cases.

\begin{CasesA}
\item ($\be=0$).
In this case, we immediately have $ \xi \in {\rm JUMP} (1+{\sf Ord})$.
Moreover, $x=1$, and $y=1$ if and only if $\al=\om^\al$.

We consider two sub-cases, depending on whether $\al<\om^\al$ or $\al = \om^\al$.
If $\al<\om^\al$, then $y=0$, and since $\be+x=1$, $\om^{\om^\al}(y+\be+x-1) =0$.
Thus we must show that $\sigma^*(\Om(1+\al)+\be)= 0$.
Note that $\nt^*(\Om(1+\al ) )=0$ if $ 1+\al $ is not of the form $\nt^*(\xi')$ for some $\xi'>\xi$.
Since $\nt$ is injective, in order to prove this it suffices to find $\xi'<\xi$ with $1+\al=\nt(\xi')$.

If $\al$ is a successor, then $1+\al=\nt(1+\al-1)$.
Otherwise, we may write $\al=\ga'+\om^{\al'} n$, where $n>0$ and $\ga'$ is either zero or of the form $\eta+\om^{\de }$ with $\de >\al'$.
Observe that $\ga'=\om^{\al'}\ga$ for some $\ga$ which is either zero or a limit, so $\al=\om^{\al'}(\ga+n)$.
Let $x'=x^\ga_{\al'}$ and $y'=y_{\al'}$.
We claim that $\beta':=-y'+\ga+n -x'$ is well-defined.
We have that $n>0$, so $ n -x'\geq 0$.
If $\ga$ is a limit or $y'=0$, then $\beta' = \ga+n -x'$.
Otherwise, $\ga=0$ and $\al'=\om^{\al'}$; but then $\al\neq \al'$ (as $\al<\om^\al$), so we cannot have $n=1$.
It follows that $ n\geq 2$, hence $\beta' = -y'+ n -x' \geq n-2$.
Moreover, $y'+\be'+x'=\ga+n$, so the induction hypothesis yields $\al=\om^\de(\ga+n)=\nt(\Om(1+\al')+\be' )$, i.e.~$\al=\nt(\xi')$ with $\xi':=\Om(1+\al')+\be'<\xi$.
We conclude that $\nt^*(\Om(1+\al ) )=0$, as needed.

If $\al=\om^\al$, recalling that $\be = 0$ and $x=1$, we have that $y=1$ and $\al=\om^{\om^\al}=\om^{\om^\al}(y+\be +x-1)$.
Therefore, in this case we must show that $\sigma^*(\Om(1+\al) )=\al$, which holds when $\al=\nt(\xi')$ for some $\xi'>\xi$.
Since $\nt$ is surjective outside of zero, this means that $\al\neq \nt(\xi')$ if $\xi'<\xi$.
Thus we will prove the latter.

First note that $\al$ cannot be a successor, so $\al\neq\nt(\xi')$ if $\xi'$ is countable.
Moreover, $1+\al=\al=\om^{\om^\al}$ and $\om^{\om^\al}>\om^{\om^{\al'}}(y'+\be'+x')$ whenever $\al',\be'< \al$.
In view of the induction hypothesis, this rules out $1+\al=\nt(\xi')$ if $\xi>\xi'=\Om(1+\al')+\be'$.
Thus indeed $ 1+\al \neq \nt (\xi')$ for any $\xi'<\xi$ and $\sigma^*(\Om(1+\al) )=\al$.
In either case, the claim holds.

\item ($\be\in\rm Succ$). 
Once again, we immediately have $ \xi \in {\rm JUMP} (1+{\sf Ord})$, and
\[\nt ^*(\xi) = \nt(\Om(1+\al)+\be-1) \stackrel{\text{\sc ih}} = \om^{\om^\al} (y+\be+x-1),\]
since the values of $x,y$ are unchanged.

\item ($\be\in\rm Lim$). We claim that $\be\neq \nt(\xi')$ for any $\xi'<\xi$.
Any $\xi'<\xi$ is either countable, so $\nt(\xi')$ is a successor and $\be\neq \nt(\xi')$, or else it is of the form $\Om(1+\al')+\be'<\Om(1+\al)+\be$.
In this case $\al'\leq \al$ and if $\be'\geq \be$ it follows that $\nt(\xi')>\be$, while if $\be'<\be $,
\[ \nt(\Om(1+\al')+\be') \leq \nt(\Om(1+\al) +\be')  \stackrel{\text{\sc ih}}\leq  \om^{\om^{\al}}(1+\be'+1) < \om^{\om^{\al}}\be = \be.\]
It follows that $\be\neq \nt(\xi')$ for any $\xi'<\xi$.
Given that $\be\in \nt(\ve_{\Om+1})\cap (1+{\sf Ord})$, $\be = \nt(\xi')  $ for some $\xi'>\xi$, and $ \xi \in {\rm JUMP} (1+{\sf Ord})$ as claimed.
It moreover follows using the induction hypothesis that
\[\nt^*(\xi) = \be =\om^{\om^\al}\be = \nt(\Om(1+\al)+\be-1). \]
This covers all cases and concludes the proof of Claim \ref{claimJump}.\qedhere
\end{CasesA}
\end{proof}

With this claim at hand, we consider cases according to the shape of $\al$.

\begin{Cases}
\item ($\al = 0$).
In this case, we have that
\begin{align*}
\cfs{\tv\zeta}m_\vartheta & = \cfs{\om^{\om^0} (y+\be+x)}m_\vartheta\\
& =  \cfs{\big ( \om  (y+\be+x-1) + \om \big )}m_\vartheta\\
& = \om  (y+\be+x-1) +   m.
\end{align*}
Meanwhile,
\[\cfs\zeta {m+1}_\nt = \nt ( \cfs\zeta m_\nt ) =  \cfs\zeta m_\nt + 1.\]
By a subsidiary induction on $m$, we prove that 
\[\cfs\zeta {m }_\nt\geq  \om  (y+\be+x-1) +  m.\]
For $m=0$, we have that
\[\cfs\zeta {0}_\nt = \nt^*(\zeta) = \om  (y+\be+x-1)+0.\]
For the inductive step, we merely note that $\cfs\zeta {m+1}_\nt > \cfs\zeta {m }_\nt$, hence 
\[\cfs\zeta {m+1}_\nt \geq \cfs\zeta {m }_\nt+1\stackrel{\text{\sc ih}}\geq \om (y+\be+x-1)+m+1,\]
as needed.

\item ($\al\in \rm Lim$).
In this case, $1+\al=\al$.
We claim that $\cfs{ \al } {m+1}_\nt \geq 1+\cfs\al  m_\vartheta$.
If $\cfs{ \al } {m}_\vartheta $ is infinite then this follows from the induction hypothesis, since 
$\cfs{ \al } {m+1}_\nt \geq   \cfs\al  m_\vartheta = 1+ \cfs\al  m_\vartheta$.
If $\cfs\al  m_\vartheta\leq 1$, we note that since fundamental sequences are strictly increasing, 
$\cfs{ \al } {m+1}_\nt \geq m+1\geq 2$.
Finally, if $1<\cfs\al  m_\vartheta <\om$, Lemma \ref{lemmFinVal} yields $\al=\om$, so that $\cfs\al  m_\vartheta  = m$.
But once again we use $\cfs{ \al } {m+1}_\nt \geq m+1$ to obtain the desired inequality.

Then,
\begin{align*}\cfs\zeta {m+1}_\nt& = \cfs{\nt(\Om \al +\be ) }{m+1} _\nt
 =  \nt   \big( \Om\cfs{ \al } {m+1}_\nt + \nt^*(\xi) \big  ) \\
 &  \geq  \nt  \big   (\Om(1+\cfs\al  m_\vartheta) + \om^{\om^\al} (y+\be+x-1) \big   )\\
 & \stackrel{\text{\sc ih}}  =   \om^{\om^{\cfs \al m_\vartheta}} \cdot  \big ( y' + \om^{\om^\al} (y+\be+x-1  ) + x' \big )\\
\end{align*}
for  $  y' =y_{\cfs\al m_\vartheta}$ and $x'=x^{\om^{\om^\al} (y+\be+x-1  )}_{\cfs\al m}$.
Now, note that
\[\om^{\om^{\cfs \al m_\vartheta}} \cdot \om^{\om^\al} (y+\be+x-1  ) = \om^{\om^\al} (y+\be+x-1  ),\]
so $x'=1$.
Hence,
\begin{align*}\cfs\zeta {m+1}_\nt& \geq  \om^{\om^{\cfs \al m_\vartheta}} \cdot  \big ( y' + \om^{\om^\al} (y+\be+x-1  ) + x' \big )\\
&\geq \om^{\om^{\cfs \al m_\vartheta}} \cdot  \big (  \om^{\om^\al} (y+\be+x-1) +1 \big )  \\
&  =     \om^{\om^\al} (y+\be+x-1 )   + \om^{\om^{\cfs \al m_\vartheta}} .
\end{align*}

Meanwhile,
\begin{align*}
\cfs{\tv \zeta}m_\vartheta & = \cfs{\big (\om^{\om^\al}(y+\be+x) \big )}m_\vartheta \\
&  \leq     \om^{\om^\al} (y+\be+x-1 )   + \om^{\om^{\cfs \al m_\vartheta}} 
\end{align*}
(where the last inequality could be strict if $\al=\om^\al$ but $\cfs\al m_\vartheta<\om^{\cfs \al m_\vartheta}$).
The claim follows.

\item ($\al =\eta+1 \in \rm Succ$).
In this case, we have that
\begin{align*}
\cfs{\tv\zeta}m_\vartheta & = \cfs{\om^{\om^\al} (y+\be+x)}m_\vartheta\\
& =  \cfs{\big ( \om^{\om^\al} (y+\be+x-1) + \om^{ \om^{ \eta +1}} \big )}m_\vartheta\\
& = \om^{\om^\al} (y+\be+x-1) + \om^{\om^\eta \cdot m},
\end{align*}
and also
\[\cfs\zeta {m+1}_\nt = \nt (\Om(1+\eta) + \cfs\zeta m_\nt ) = \om^{\om^\eta } (y'+\cfs\zeta m_\nt+x')\]
for suitable $x',y'$.
By a subsidiary induction on $m$, we prove that 
\[\cfs\zeta {m+1}_\nt\geq  \om^{\om^\al} (y+\be+x-1) + \om^{\om^\eta \cdot m}.\]
For $m=0$, first we note that by Claim~\ref{claimJump}, $\cfs\zeta {0}_\nt = \nt^*(\zeta) = \om^{\om^\al} (y+\be+x-1)$, hence since $\cfs\zeta { 1}_\nt > \cfs\zeta {0}_\nt$ we see that
\[\cfs\zeta { 1}_\nt \geq   \om^{\om^\al} (y+\be+x-1)+1 = \om^{\om^\al} (y+\be+x-1) + \om^{\om^\eta \cdot 0}. \]
For the inductive step, we see that
\begin{align*}
\cfs\zeta {m+2}_\nt & =  \om^{\om^\eta } (y'+\cfs\zeta {m+1}_\nt+x')\\
&\stackrel{\text{\sc ih}}\geq \om^{\om^\eta } (y'+ \om^{\om^\al} (y+\be+x-1) + \om^{\om^\eta \cdot m} +x')\\
& \geq   \om^{\om^\al} (y+\be+x-1) + \om^{\om^\eta \cdot (m+1)},
\end{align*}
as needed.
\end{Cases}

\end{Cases}

\item ($\xi\geq \Om^2$).
In this case, we have that $\tv\zeta = \vartheta(\Om+\xi' )$.
Note that by Lemma~\ref{lemEpsilon}, $\xi\geq \Om $ implies that $\vartheta(\xi) = \om^{\vartheta(\xi)}$, meaning that we must prove that \eqref{eqSide} holds.

Write $\xi = \Om^2+\ga $ and further write $\ga=\al+\be$ in such a way that $\check \xi = \Om^2+\al $, so that $\be<\Om$ is possibly zero.
Observe that all infinite coefficients of $\xi $ are also coefficients of $\ga$.
Thus, either $(\Om^2+\ga)^*=(\Om +\ga)^*$ or both are finite, and similarly $\tau (\Om^2+\ga) =\tau (\Om +\ga)$  if they are infinite.

Let $v=1$ if the $\Om$-normal form of $\xi$ begins with a term $\Om^2k$ with $k\in (0,\om)$, and otherwise let $v=0$.
Similarly, let $w=1$ if the $\Om$-normal form of $\Om+\ga$ begins with a term $\Om k$ with $k\in (0,\om)$, and otherwise let $v=0$.
Note that $\Om^2+\ga=\Om^2v+\ga $ and similarly $\Om +\ga=\Om w+\ga $.
The crucial observation here is that  if $\ga>0$, then by Lemma~\ref{lemmFSSum}, we have that for all $\chi$, $\fs{(\Om^2v+\ga)}\chi =\Om^2v+\fs{\ga}\chi $, and similarly  $\fs{(\Om w+\ga)}\chi =\Om w+\fs{\ga}\chi $, with the same observation holding for $\al$ in place of $\ga$.
In particular, if $\ga>0$, then $\fs{(\Om^2v+\ga)} 1^*< (\Om^2v+\ga) ^*$ iff $\fs{(\Om w+\ga)} 1^*< (\Om w+\ga) ^*$, except possibly in cases where the maximal coefficients are finite.

\begin{claim}
$\nt^*(\Om^2+\ga)= \vt^*(\Om +\ga)$.
\end{claim}

\begin{proof}[Proof of claim]
If $\nt^*(\Om^2+\ga) > 0$ then by the definition of $\nt^*$, $\nt^*(\Om^2+\ga) = \nt (\Om^2+ \tilde\ga) $ for some $\tilde \ga$ with $\tilde \ga+1\geq \ga$.
Then, the induction hypothesis (for $\nt (\Om^2+ \tilde\ga)<\zeta$) yields $\nt (\Om^2+ \tilde\ga) = \vt (\Om + \tilde\ga)$.
We have that either $\nt^*(\Om^2+\ga)=\tau( \Om^2+\ga )$ or $\nt^*(\Om^2+\ga)= \nt (\Om^2+\ga-1)$ when $\ga$ is a successor; in either case, as we have noted that the two share maximal coefficients and terminal parts, it can be checked that $\nt^*(\Om^2+\ga)=\vt^*(\Om +\ga)$, as claimed.

The argument is analogous if $\vartheta ^*(\Om +\ga) > 0$: $\vartheta ^*(\Om +\ga)=\vartheta{(\Om+\tilde\ga)}$ for some $\tilde\ga$ with $\tilde\ga+1\geq \ga$, and by induction we have that $ \nt(\Om^2+\tilde\ga ) = \vartheta{(\Om+\tilde\ga)}$.
Reasoning as above, it follows that $\nt^*(\Om^2+  \ga) = \vt^*(\Om +  \ga) $.

Otherwise, $\nt^*(\Om^2+\ga) = 0 = \vt^*(\Om +  \ga)$, so equality holds in all cases.
\end{proof}

Next we consider two sub-cases.
\begin{Cases}

\item ($\check \xi <\Om^2+\Om$).
Recall that $\check \xi=\Om^2+\al$.
Then, $\al<\Om$ and $\al$ is either zero, a successor, or $\al = \nt(\Om^2+\tilde\xi)$ for some $\tilde\xi>\al$.
Let $\theta = \nt^*(\xi)  $ and let $z=1$ if $\theta>0$, otherwise $z=0$.

By a simple computation using Lemma~\ref{lemmThetaExp}, $\cfs{\vartheta(\Om +\ga)} m_\vartheta = \om_m( \theta+ z )$.
Recall that we need to establish \eqref{eqSide}, for which it suffices to show that $\cfs\zeta {m +1}_\nt \geq  \om_{m +1}(\theta +1 )$ for all $m> 0$.
We first consider the additional case of $m=0$.
We compute $\cfs\zeta {0 } _\nt = \theta$, so that $\cfs\zeta { 1 } _\nt = \sigma(
\Om\theta) = \om^{\om^\theta}(y_\theta+1)$ (since $x^0_\theta=1$).
If $\theta=0 $ then  $\om^{\om^\theta}> \theta+1 $, and otherwise by Lemma~\ref{lemEpsilon} we see that $\theta=\om^\theta$, so $y_\theta=1$ and $\om^{\om^\theta}(y_\theta+1) >\theta+1 =\om_0(\theta+1)$.
Thus for the inductive step we may uniformly assume that $m\geq 0$ and $\cfs\zeta {m +1}_\nt \geq  \om_{m  }(\theta +1 )$ to obtain
\begin{align*}
\cfs \zeta {m+2} _\nt & = \nt(\Om \cfs\zeta {m+1}_\nt   ) =\om^{\om^{\cfs\zeta {m+1}_\nt }}\\
& \stackrel{\text{\sc ih}} \geq\om^{\om^{\om_{m }(\theta +1 ) }} = \om_{m+2}(\theta +1).
\end{align*}
We conclude that $\cfs \zeta {m+1} _\nt\geq\om^{\cfs{\tv\zeta}m_\vartheta} $, as needed.

\item ($\check \xi \geq \Om^2+\Om$).
Recall that $\theta:=\vartheta^*( \Om w+\ga) = \nt^*( \Om^2v+\ga )$ and $\check \xi = \Om^2+\al$ and consider two sub-cases.
\begin{Cases}

\item ($ \tau<\Om$).
We have that
\begin{align*}
\cfs\zeta{m+1}_\nt & = \nt(\Om^2v+\fs {\al} {\cfs\tau {m+1}_\nt} +\theta)\\
 & \stackrel{\text{\sc ih} } \geq \nt(\Om^2v+\fs {\al} {\cfs\tau m_\vartheta} +\theta)\\
  & \stackrel{\text{\sc ih} } = \vartheta(\Om w +\fs {\al} {\cfs\tau m_\vartheta} +\theta)\\
   &  =\om^{ \vartheta(\Om w +\fs {\al} {\cfs\tau m_\vartheta} +\theta)}= \om^{\cfs{\tv\zeta}m_{\vartheta}}.
\end{align*}

\item ($ \tau=\Om$).
We prove by a subsidiary inducion on $m$ that
\[\cfs\zeta{m+1}_\nt \geq   \cfs{\tv\zeta}{m+1}_{\vartheta} = \om^{\cfs{\tv\zeta}{m+1}_{\vartheta} }.\]
We first observe that $\cfs\zeta{0}_\nt = \theta=   \cfs{\tv\zeta}{0}_{\vartheta}$, so we may assume uniformly that $\cfs\zeta{m }_\nt \geq   \cfs{\tv\zeta}{m }_{\vartheta}$.

Then,
\begin{align*}
\cfs\zeta{m+1}_\nt & = \nt(\Om^2 v +\fs \al {\cfs\zeta{m }_\nt })
 \stackrel{\text{\sc ih}(m)} \geq  \nt(\Om^2 v +\fs \al { \cfs{\tv\zeta}{m }_{\vartheta}
 })\\
  & \stackrel{\text{\sc ih}(\zeta)} =  \vartheta (\Om w +\fs \al { \cfs{\tv\zeta}{m }_{\vartheta}
 })
   = \om^{ \vartheta(\Om w +\fs \al { \cfs{\tv\zeta}{m }_{\vartheta}
 })}\\
 & = \om^{\cfs{\tv\zeta}{m+1 }_{\vartheta}}.
\end{align*}
\end{Cases}
\end{Cases}
\end{Cases}
This covers all cases and concludes the proof.
\end{proof}

\section{The Bachmann property}\label{secBachmannP}
In this section we show that our systems of fundamental sequences enjoys the Bachmann property.
For the proof by induction to work, we need to prove a somewhat more general property.
For this, we extend the definition of $\cfs\cdot\cdot$ to set $\cfs\xi n = \fs \xi 0$ if $\tau(\xi) = \Om$; this is merely a technical artifice to avoid treating the case $\tau(\xi) = \Om$ separately.

\begin{theorem}[Bachmann property]\label{theoBach}
Let $\mathbb X$ be either $1+\sf Ord$ or $\mathbb P$ and let $\gt=\gt_\mathbb X$.
Let $\al,\ga \in \constr$ and $x>0$ be such that $\tau(\al)<\Om$ and $\cfs\al x<\ga<\al $.
Then, $\cfs\al x\leq \nfs\ga 1$.
\end{theorem}

\begin{proof}
For our induction to go through, we need to prove a slight extension: namely, the theorem should also hold for $x=0$ whenever $\tau(\al)=\gt(\xi)$ with $\tau(\check \xi)=\Om$.
We proceed by induction on $(x,\al^*,\al,\ga^*,\ga)$, ordered lexicographically, to show that under these extended conditions, if $\cfs\al x<\ga<\al $, then $\cfs\al x\leq \cfs\ga 1$.

Clearly $\al\neq 0$ and $\al$ cannot be a successor, while the claim follows trivially if $\ga =\delta+1$ is a successor, as in this case $\nfs \ga 1 = \delta \geq \cfs \al x$.
We may also assume that $\cfs\al x> 0$, for otherwise the claim is trivial.
We consider other cases below.
\begin{Cases}
\item ($\al<\Om$).
Note that we have $\tau(\ga) =\ga <\al <\Om$.
This case is sub-divided as follows.
\begin{Cases}
\item ($\al\notin\mathbb X$).
This is only possible if $\mathbb X=\mathbb P$, in which case $\al = \om^\eta+\xi$ in Cantor normal form with $\xi>0$.
Then,
\[\om^\eta +\cfs \xi x =  \cfs \al x <\ga <   \al  = \om^\eta + \xi .\]
This can only happen if $\ga=\om^\eta+\delta$ with $\cfs \xi x < \delta < \xi$, which by the induction hypothesis yields $\cfs\xi x \leq\cfs\de 1$ and hence $\cfs \al x \leq \cfs\ga 1 $.

\item ($\al\in \mathbb X\setminus \mathbb X'$).
Again, this is only possible if $\mathbb X=\mathbb P$ (since  $\al$ cannot be a successor), and by Proposition \ref{propNotFix}, $\al=\vartheta(\xi)$ with $\xi\in {\rm JUMP}(\mathbb P)\cap \Om$.
In this case, we may assume that $x$ is chosen so that $\cfs \al x <\ga\leq \cfs\al{x+1}$ (otherwise, replace $x$ by a larger value).
Then,
\[\vartheta^*(\xi) \cdot x =\cfs \al x <\ga \leq \cfs\al{x+1} =\vartheta^*(\xi)\cdot(x+1).\]
This implies that $\ga = \vartheta^*(\xi) \cdot x   +\de$ with $0<\de <\vartheta^*(\xi)$.
By definition of $\vartheta^*$ we see that $\vartheta^*(\xi) = \vartheta(\eta)$ for some $\eta$, hence it is additively indecomposable and $  \vartheta^*(\xi) \cdot x +\de$ is in Cantor normal form, which implies that $\cfs \ga 1 =  \vartheta^*(\xi) \cdot x +\cfs\de 1 \geq \vartheta^*(\xi) \cdot x$, as needed.

\item ($\al \in \mathbb X'$).
Then, $\al=\gt(\xi)$ for some $\xi<\ve_{\Om+1}$, and either $\xi\geq \Om$ or $\xi \in {\rm Lim} \setminus {\rm JUMP}(\mathbb X)$.
Let $\tau = \tau(\check \xi)$.
In order to unify some cases below, if $\tau<\Om$, define $\tau_n = \cfs\tau n$ and $\theta =\gt^*(\xi)$.
If $\tau =\Om $, then $\tau_n = \gt^{(n)}(\xi)$ and $\theta=0$.
Then, $\cfs\al x = \gt\big (\fs{\check \xi}{\tau_n}+\theta \big)$.
Consider the following sub-cases.

\begin{Cases}

\item ($\ga\notin\mathbb X$).
This case is only possible if $\mathbb X=\mathbb P$.
Write $\ga=\om^\eta+\de$ with $\de>0$ in Cantor normal form.
Since $\al\in\mathbb P'$, $\cfs\al x \in \mathbb P$, hence $\cfs\al x<\ga$ yields $\cfs\al  x\leq \om^\eta \leq \cfs \ga 1$.

In subsequent cases, we have that $\ga\in\mathbb X$ and can write $\ga=\gt(\zeta)$.

\item ($\ga \in \mathbb X\setminus \mathbb X'$).
Since we are assuming that $\ga$ is not a successor, we must have $\mathbb X=\mathbb P$ and $\cfs \ga n = \vartheta^*( \zeta)\cdot n$ for all $n$.
Since $\cfs\ga n\to \ga$, we can find $n$ so that $\cfs\al x <  \vartheta^*( \zeta) \cdot n$, but $\cfs\al x$ is additively indecomposable so $\cfs\al x\leq \vartheta^*( \zeta) =\cfs \ga 1 $.

\item ($\ga\in \mathbb X'$).
First note that $\zeta<\xi$, for otherwise $\gt(\zeta)<\gt(\xi)$ and Lemma \ref{lemmBackwardsExt} would yield $\ga \leq \cfs\al x $ (this includes the extended case where $x=0$ and $\tau(\check \xi)=\Om$).
We also cannot have $\zeta+1=\xi$, since in this case $\cfs\al x \geq \gt^*(\xi)=\gt(\zeta)$, so $\zeta+1<\xi$.
Similarly, if $\fs{\check \xi}{\tau_x}+\theta >\zeta $ then $\gt \big (\fs{\check \xi}{\tau_x}+\theta  \big ) \leq \cfs \ga 1$, so we may assume that $\fs{\check \xi}{\tau_x}+\theta<\zeta $.

We claim moreover that $ \zeta < \check \xi$.
If this were not the case, we would have $\zeta=\check \xi +\delta$ and $\xi=\check \xi +\beta$ for some $\delta <\beta$, hence $\beta\neq 0$ and $ \gt^*(\xi)\neq 0$.
By definition of $\gt^*$, we have that $\gt^*(\xi) =\gt(\tilde\xi)$ for some $\tilde\xi$ with either $\tilde\xi+1 = \xi$, or $\tilde\xi>\xi$ and $\gt (\tilde \xi) = \be > \zeta^*  $. In either case, $  \gt (\tilde \xi) \geq \gt (\zeta)$, hence $\cfs\al 1 > \gt^*(\xi) =\gt(\tilde \xi) \geq \gt (\zeta)$, contrary to our assumption.

Since $\zeta<\check \xi$, it follows that $\check \zeta<\check \xi$; since $\check \xi\neq 0$, it must be a limit, so we also have that $\zeta+1< \check \xi$.
With this in mind, consider the following cases.

\begin{Cases}

\item ($\check \zeta = \fs{\check\xi} {\tau_x}$).
Then, $\zeta = \fs{\check\xi} {\tau_x} +\delta$ for some $\delta$, and since $\gt(\zeta)>\gt\big ( \fs{\check\xi} {\tau_x} +\theta \big )$, we obtain $\delta>\theta$.
We claim that $\gt^*(\zeta)>\cfs \al x$.
Note that $\zeta\neq \check \zeta$, so $\delta$ is either a successor or of the form $\gt(\zeta')$ for some $\zeta'>\zeta$.
If $\delta =\delta'+1$, we have that $ \gt^*(\zeta) = \gt(\fs{\check\xi} {\tau_x} +\delta')\geq \gt(\fs{\check\xi} {\tau_x} +\theta)$.
If $\delta = \gt(\zeta')$ for some $\zeta'>\zeta$, then also $\zeta'> \fs{ \check\xi} {\tau_x}  +\theta$, and moreover $\delta> \check \zeta^* = \fs{ \check\xi} {\tau_x}^*$ implies that
\[\big (\fs{ \check\xi} {\tau_x}  +\theta \big )^* < \delta = \gt (\zeta'),\]
so that $\gt^*(\zeta) =\gt(\zeta') >  \gt  ( \fs{\check\xi} {\tau_x} +\theta   ) = \cfs \al x$, as claimed.
We conclude that
\[\cfs\ga 1 > \gt^*(\zeta) \geq \cfs\al x. \]

\item ($\check \zeta > \fs{\check\xi} {\tau_x}$).
In this case, by the assumption that $\ga\in\mathbb X'$ we can write $\cfs \ga 1 = \gt(\fs{\check\zeta}\rho+\delta )$ for suitable $\rho,\delta<\Om$ with $\fs{\check\zeta}\rho+\delta <\zeta$.
 
We first claim that $\gt^*(\xi) < \cfs\ga 1$.
If $\gt^*(\xi)=0$ this is obvious, otherwise we show that $\gt^*(\xi) =\gt(\xi')$ for some $\xi'>\zeta$.
This suffices, as $\gt^*(\xi) \leq \cfs\al x <\gt(\zeta)$, so Lemma \ref{lemmBackwardsExt} yields $\gt^*(\xi)  <\cfs \ga 1$.

To find such a $\xi'$, if $\xi$ is a successor, let $\xi' $ be its predecessor.
We have observed that $\zeta+1< \check \xi$, so  $\xi'>\zeta$.
Otherwise, $\xi\in \rm Lim$ and $\gt^*(\xi) \neq 0$.
From $\gt^*(\xi)\neq 0$ we obtain $\gt^*(\xi) = \gt(\xi')$ for some $\xi'>\xi>\zeta$ by definition of $\gt^*(\xi)$, as required.

In case that $\tau  = \Om$ and $x=0$, we already obtain $\cfs \al 0 = \gt^*(\xi) < \cfs \ga 1$, so we may henceforth assume $x>0$.
In other cases, define $\pi$ to be $\tau(\check \xi)$ if $\tau(\check \xi)<\Om$, and $\pi=\al$ otherwise.
We claim that $\zeta< \fs {\check \xi}{\pi}$.
If $\tau(\check \xi)<\Om$, this follows simply because $\check \xi = \fs {\check \xi}{\pi}>\zeta$, and if $\tau(\check \xi)<\Om$, we note that if $\zeta\geq \fs {\check \xi}{\pi}$, then by Lemma \ref{lemmBoundMC}, $\fs {\check \xi}{\pi} < \zeta < \check \xi $ implies that $\zeta^*\geq \pi=\al$, and hence $\gt(\zeta)>\al$, a contradiction.

Since $\check \zeta\leq \zeta$, we thus have that $\fs{\check \xi}{\tau_x } <\check \zeta < \fs {\check \xi}{\pi}$.
It is not hard to check that $\fs{\check \xi}{\tau_x } = \cfs{ \fs {\check \xi}{\pi}} y$ for suitable $y$; namely, $y=x$ if $\pi=\tau<\Om$ (as $\tau_n$ was defined to be $\cfs \tau n$), and $y=x-1$ if $\pi=\al$ and $\tau=\Om$, again by our definition of $\tau_n$.

We wish to use our induction hypothesis to obtain $\fs{\check \xi}{\tau_x } = \cfs {\fs{\check \xi}{\pi}} y \leq \nfs{{\check\zeta}} 1$ from  $\cfs{ \fs {\check \xi}{\pi}} y <\check \zeta < \fs {\check \xi}{\pi}$.
Let us check that indeed the hypothesis is available according to our lexicographic ordering of the variables.
If $\tau<\Om$, $\pi=\tau<\alpha  $, and since $\al=\al^* > \fs{\check \xi}{\tau }^*$, we apply the induction hypothesis to the second variable in our lexicographic ordering.
If $\tau=\Om$, then $y<x$ and we apply induction to the first variable in the ordering (note that here we may apply the extended case for $y=0$).
Thus we indeed have that $\fs{\check \xi}{\tau_x }   \leq \nfs\zeta 1$.

Next we check that $\cfs{\check\zeta} 1 \leq \fs{\check \zeta} \rho $.
If $ \tau(\check \zeta) =\Om$, then we have defined $\cfs{\check\zeta} n = \fs{\check\zeta} 0 $ for all $n$, so this is clear.
Otherwise, we have from the definition of the fundamental sequences that $\cfs\gamma 1 = \gt(\cfs {\check \zeta} 1 + \gt^*(\zeta))$, so that $\cfs{\check\zeta} 1 = \fs{\check\zeta} {\rho}$.
We thus obtain
\[\fs{\check \xi}{\tau_x } \leq \cfs{\check\zeta} 1 \leq \fs{\check \zeta} \rho  <\fs{\check \xi}\pi,\]
which yields $\fs{\check \xi}{\tau_x }^* \leq \fs{\check\zeta} \rho ^*$ by Lemma \ref{lemmBoundMC}.
Moreover, we have proven that $\gt^*(\xi) < \cfs\ga 1 = \gt(\fs{\check \zeta}\rho+\delta)$, which since $\theta\leq \gt^*(\xi)  $ gives us $\big (\fs{\check \xi}{\tau_x}+\theta \big)^* < \gt(\fs{\check \zeta}\rho+\delta)$.
It remains to prove that $ \fs{\check \xi}{\tau_x}+\theta  \leq   \fs{\check \zeta}\rho+\delta $ to obtain $\gt \big (\fs{\check \xi}{\tau_x}+\theta \big ) \leq \gt \big (\fs{\check \zeta}\rho+\delta\big )$.

If $\theta=0$ or $ \fs{\check \xi}{\tau_x} <  \fs{\check \zeta}\rho$, this is immediate.
Otherwise, $\fs{\check \xi}{\tau_x} ^* < \theta = \gt^*(\xi) =\gt(\xi')$ for some $\xi'>\zeta$, but we have seen that $\gt^*(\xi) <\cfs \ga 1 =  \gt \big (\fs{\check \zeta}\rho+\delta\big ) $.
This can only occur if $ \big (\fs{\check \zeta}\rho+\delta\big )^* \geq \gt(\xi')$; but $\fs{\check \zeta}\rho^*= \fs{\check \xi}{\tau_x} ^* < \gt(\xi')$, so we conclude that $\delta\geq \gt(\xi')= \theta$, yielding the desired result.
\end{Cases}

\end{Cases}

\end{Cases}

\item ($\al>\Om$).
Write $\al = \Om^\eta \be +\la$ in $\Om$-normal and consider several sub-cases.
\begin{Cases}
\item ($\la>0$).
Then $\cfs\al x= \Om^\eta\be + \cfs \la x$, and $\cfs \al x <\ga\leq \al = \Om^\eta \be+\la $ yields $\ga = \Om^\eta \be +\zeta$ for some $\zeta$ with $\cfs \la x <\zeta\leq \la $.
The induction hypothesis implies that $\cfs \la x \leq \nfs\zeta 1 $, hence
\[\cfs \al x = \Om^\eta\be+\cfs \la x \leq \Om^\eta\be+\nfs \zeta 1 =\nfs \ga 1. \]

\item ($\la=0$).
We consider sub-cases according to the shape of $\be$.

\begin{Cases}
\item ($\be=\xi+1$).
Since we are assuming that $\tau(\al)<\Om$, we must have that $\eta \in \rm Lim$.
Then, $\cfs \al x= \Om^{\eta}\cdot \xi + \Om^{\cfs \eta x} < \ga< \al =\Om^{\eta} \cdot (\xi+1) $.
Therefore, we can write
\[\ga=\Om^{\eta}\cdot \xi +\Om^{\chi}\rho + \delta \]
in $\Om$-normal form with $\Om^{\chi}\rho + \delta$ also in $\Om$-normal form, where $\rho>0$, $  \cfs\eta x\leq \chi <\eta$, and either $  \cfs\eta x < \chi$, $\rho>1$, or $\de>0$.
Consider the following sub-cases.
\begin{Cases}
\item ($\de>0$).
Then,  $\nfs\ga 1 =  \Om^{\eta}\cdot \xi +\Om^{\chi}\rho +\cfs\de 1 \geq \cfs \al x$, as required.\footnote{In this case, it is possible to have $\tau(\delta)=\Om$, and we recall that $\nfs \de 1 := \fs\de 0$.
Note that the proof itself is unaffected, as we only need $\nfs \de 1\geq 0$, which always holds. A similar comment applies to some of the subsequent cases.}

\item ($\rho > 1$).
In this case, $\nfs\ga 1 =  \Om^{\eta}\cdot \xi +\Om^{\chi}\cfs \rho 1  \geq \cfs \al x$, where the inequality follows from $\cfs \rho 1 = \rho-1\geq 1$ when $\rho$ is a successor and from $\cfs \rho 1 > \cfs \rho 0 \geq 0  $ when $\rho$ is a limit.

\item 
($  \cfs\eta x < \chi$, $\rho=1$ and $\de=0$).
In this case, $\nfs \ga 1 = \Om^{\eta}\cdot \xi +\Om^{\nfs \chi 1} $.
Since $  \cfs\eta x < \chi <\eta $, the induction hypothesis yields $\cfs \eta x\leq \nfs\chi 1$, which implies that $\cfs \al x\leq \nfs \ga 1$.

\end{Cases}

\item
($\be\in \rm Lim$).
In this case,
\[\cfs \al x = \Om^{\eta}\cdot (\cfs\be x)<\ga< \al =\Om^{\eta}\cdot \be.\]
Write $\ga= \Om^\chi\rho+\delta$ in $\Om$-normal form.
Then, $\Om^{\eta}\cdot (\cfs\be x) < \ga < \Om^{\eta}\cdot \be $ yields $\chi = \eta$ and $\rho\geq \cfs \be x$.
Consider the following sub-cases.
\begin{Cases}
\item ($\de>0$).
In this case, $\Om^{\eta}\cdot (\cfs\be x)  \leq \Om^{\eta}\cdot \rho + \nfs \de 1 = \nfs\ga 1$.

\item ($\de=0$).
Then,  $\Om^{\eta}\cdot (\cfs\be x) < \Om^\eta \rho < \Om^{\eta}\cdot \beta$ implies that $\cfs \be x<\rho\leq  \be$, so that by the induction hypothesis $\cfs\be x \leq \cfs \rho 1$, which implies that $\Om^{\eta}\cdot (\cfs\be x) \leq  \Om^\eta \cfs\rho 1 \leq \cfs \ga 1$, as required.
\end{Cases}
\end{Cases}
\end{Cases}
\end{Cases}
This covers all cases and concludes the proof.
\end{proof}

    \section{Notation systems and norms}\label{secNorms}
    
    The $\gt_\mathbb X$ functions are meant to play a role in ordinal notation systems among other functions which represent the `predicative' part of the system.
Given a set of functions $\mathcal F$ such that each $f\in \mathcal F$ is of the form $f\colon {\sf Ord}^n\to{\sf Ord}$ for some $n$, an {\em $\mathcal F$-term} is a formal expression built inductively from the elements of $\mathcal F$ along with some set of variables: formally, we fix a countable set $V$ of variables and each $x \in V$ is an $\mathcal F$-term, as is the formal expression $f(t_1,\ldots,t_n)$ if $f\in\mathcal F$ and each $t_i$ is an $\mathcal F$-term.
We write $t(\vec x)$ to indicate that $\vec x$ is the tuple of variables appearing in $t$, and if $\vec \xi$ is a tuple of ordinals of the same arity then $|t(\vec\xi)|$ (or simply $t(\vec\xi)$ if this does not lead to confusion) is the value obtained by evaluating $t$ with each variable $x_i$ interpreted as $\xi_i$.

\begin{definition}
A family of functions $\mathcal F$ is {\em $\mathbb X$-complete} if for every ordinal $\zeta $ there is a term $t(\vec x)$ and a tuple $\vec\xi$ of elements of $\mathbb X$ such that $\zeta=|t(\vec \xi)| $.
\end{definition}

As an example, $\mathcal F=\{0 \}$ (with $0$ being regarded as a function of arity zero) is complete for $1+{\sf Ord} $, while $\{0,+\}$ is complete for $\mathbb P$.
Similarly, $\{0,+,\om^x\}$ is complete for the class of $\ve$-numbers, etc.

Given a family of $\mathbb X$-complete functions $\mathcal F$, $\mathcal F^\Om_\mathbb X$ is the family obtained by adding $\gt_\mathbb X$ to $\mathcal F$ as well as the function $\Om^xy+z$.
When $\mathbb X$ is clear, we may write simply $\mathcal F^\Om$. 

\begin{theorem}\label{theoNotation}
If $\mathcal F$ is $\mathbb X$-complete, then for every $\xi<\gt_\mathbb X(\ve_{\Om+1})$, there is a closed $\mathcal F^\Om$-term $t$ such that $\xi = |t|$.
\end{theorem}

\begin{proof}
Proceed by induction on $\xi$.
By Corollary \ref{corOmTower}, there is $\zeta<\ve_{\Om+1}$ with $\xi=\gt_\mathbb X(\zeta)$.
Since $\zeta^*<\xi$, we can use the induction hypothesis to find a term $t_\al$ for each coefficient $\al$ occurring in $\zeta$, hence a term $t_\zeta$ with $\zeta=|t_\zeta|$.
We then have that $\gt_\mathbb X(t_\zeta)$ is the desired term for $\xi$.
\end{proof}

\begin{corollary}\label{corNotation}
Every ordinal $\xi<\vartheta(\ve_{\Om+1})$ can be written in terms of $0,\sigma$, and $\Om^xy+z$, or in terms of $0,+,\vartheta $, and $\Om^xy+z$
\end{corollary}

Notation systems may be used to naturally assign norms to ordinals.

\begin{definition}
If $\mathcal F$ is a family of functions, we define the norm $\|t\|_\mathcal F$ of an $\mathcal F$-term $t$ inductively by setting $\|f(t_1,\ldots,t_n)\|_\mathcal F = 1+  \|t_1\|_\mathcal F +\ldots +\|t_1\|_\mathcal F$.
Similarly, we define the norm of an ordinal $\al$ by letting $ \|\al\|_\mathcal F$ be the least value of $\|t\|_\mathcal F$ such that $|t|=\al$, and $ \|\al\|_\mathcal F =\infty$ if there is no such value.
\end{definition}

We often write $\|\cdot\|$ instead of $\|\cdot \|_\mathcal F$ when $\mathcal F$ is clear from context, and write $\|\cdot \|_\nt$, $\|\cdot \|_\vartheta$ when $\mathcal F$ is either $\{0\}^\Om_{1+\sf Ord}$ or $\{0,+\}^\Om_{\mathbb P}$, respectively.
The following is easy to check and we will need it to show that $\big(\vartheta(\ve_{\Om+1}) ,\cfs\cdot\cdot_\vartheta \big )$ is a regular Cantorian Bachmann system.

\begin{lemma}\label{lemCantorTwo}
Let $\vartheta(\ve_{\Om+1}) > \al_1\geq\ldots\geq \al_n $ and define
\[a = \max\{n,\|\al_1\|_\vartheta,\ldots, \|\al_n\|_\vartheta\}.\]
Then,
$ a \leq  \| \om^{\al_1} +\ldots +\om^{\al_n}\|_\vartheta = O (a^2)$.
\end{lemma}

Finally, we check that our fundamental sequences are a normed system, in the following sense.

\begin{theorem}\label{theoReg}
Let $\mathbb X$ be either $1+\sf Ord$ or $\mathbb P$ and $\mathcal F $ be either $ \{0 \}$ or $ \{0,+\}$, respectively.
Let $\|\cdot\| = \|\cdot\|_{\mathcal F^\Om}$.
Then, if $\zeta<\xi \in \constr$, $\xi$ is a limit with $\tau(\xi)<\Om$, and $\|\zeta\| \leq n   $, it follows that $\zeta \leq \cfs\xi{n }$.
\end{theorem}

\begin{proof}
By induction on $(\xi^*,\xi,\zeta^*,\zeta)$, ordered lexicographically. Note that $\zeta<\xi$ implies that $\xi\neq 0$ and the claim is trivial for $\zeta=0$, so we assume otherwise.
Note also that norms are positive, so $n>0$.
Consider several cases according to the shape of $\xi$.

\begin{Cases}

\item ($\xi<\Om $).
This case sub-divides into further sub-cases.

\begin{Cases}

\item ($\xi\notin \mathbb X$).
Then, $\mathbb X=\mathbb P$ and $\xi$ is additively decomposable.
Write $\xi = \om^\al+\be$ and $\zeta =  \om^\ga+\de $ in Cantor normal form.
If $\ga<\al$, then it is readily checked that $\zeta< \cfs\xi 0 $, so we may assume that $\ga=\al$.
But then, the induction hypothesis yields $\be<\cfs\de n$, so that $\zeta<\cfs\xi n$.

\item ($\xi\in\mathbb X \setminus  \mathbb X'$).
Then, $\mathbb X=\mathbb P$ and $\xi=\gt(\al)$, with $\al<\Om$ and $\al\in {\rm JUMP}(\mathbb P)$.
In this case, $\cfs \xi n = \vartheta^*(\al)\cdot n$.
Consider the following cases.
\begin{Cases}

\item ($\zeta \in \mathbb P$).
Write $\zeta=\vartheta(\beta)$.
From $\zeta<\xi$ we obtain $\zeta< \vartheta^*(\al)\cdot m$ for some $m$, but $\zeta$ is additively indecomposable, so $\zeta\leq \vartheta^*(\al) =\cfs\xi 1$.

\item ($\zeta\notin\mathbb P$).
Write $\zeta=\om^\ga m +\de$ in Cantor normal form with $\de<\om^\ga$.
We have that $  \|\zeta\| \geq \|\om^\ga m\|+1$, so that $\|\om^\ga m\|\leq n-1$ and the induction hypothesis yields $\om^\ga m \leq \vartheta ^*(\xi) \cdot (n-1)$; since $\om^\ga$ is additively indecomposable, this implies that $\om^\ga \leq \vartheta ^*(\xi) $.
If $\om^\ga < \vartheta ^*(\xi) $, then also $\zeta < \vartheta^*(\xi) $, since $\vartheta ^*(\xi)$ is additively indecomposable. Otherwise, we have that $\de < \om^\ga = \vartheta^*(\xi)$ and $\zeta =\om^\ga m+\de < \vartheta^*(\xi) \cdot (n-1) + \vartheta^*(\xi)  = \cfs\xi n$.
\end{Cases}

\item ($\xi=\gt(\al) \in \mathbb X'$).
In this case, we may assume that $\zeta = \gt(\beta)$: this is always the case for $\mathbb X=1+\sf Ord$, and for $\mathbb X=\mathbb P$, we have that $\cfs\xi n$ is additively indecomposable, hence $\ga+\de< \cfs\xi n$ if and only if $\ga<\cfs \xi n$ and $\de <\cfs \xi n$, so we may assume that $\zeta$ is additively indecomposable as well.
With this in mind, consider the following cases.
\begin{Cases}

\item ($\zeta = \gt(\beta)$ with $\be<\al$).
We divide into further sub-cases.
\begin{Cases}

\item ($\beta\geq \check \al$).
We have that $\be = \check\al+\de$ and $\al=\check \al+ \theta$ with $\de<\theta \leq \gt^*(\al)$.
Then, $\de< \gt^*(\al)$, hence $ \gt^*(\al) \neq 0$ and $ \gt^*(\al) = \gt(\tilde\al)$ for some $\tilde\al$ with $\tilde\al+1\geq \al$, from which it follows that $\tilde\al \geq \be$.
We moreover see that $\be^* \leq \al^* \leq \gt(\tilde\al)$, and if $\be^* = \al^* $, then $\theta \leq \check \al^*$ (otherwise, $\al^* = \theta>\max\{\check \al^*,\de\}=\be^*$), which by definition of $\gt^*$ means that $\al$ must be a successor and thus $\al^* < \gt(\tilde\al)$.
In any case $\be^* < \gt(\tilde\al)$, which together with $\be\leq \tilde \al$ implies that $\gt(\be) \leq \gt^*( \al) $.
In either case,  $\gt(\be) \leq \gt^*( \al) <\cfs\al n$.

\item ($\beta< \check \al$).
Note that the induction hypothesis yields $\be^* \leq \cfs \xi {n-1}$.
Define $\tau:=\tau(\check \al)$ and consider the following cases.

\begin{Cases}

\item ($\tau < \Om $).
Then, $\cfs\xi n= \gt \big ( \cfs{\check \al}n + \gt^*(\al ) \big )$ and $\beta <\cfs{\check \al} n $ by the induction hypothesis.
Since
$\beta^* < \cfs \xi n$,
\[\zeta=\gt(\beta)< \gt \big ( \cfs{\check \al}n + \gt ^*(\al) \big ) = \cfs\xi n.\]

\item ($\tau=\Om$).
If $\be<\Om$, the induction hypothesis implies that $\beta \leq \cfs\xi {n-1} \leq \fs{\check\al}{\cfs\xi {n-1} } ^* $, hence
\[\zeta=\gt(\beta)<\gt(\fs{\check\al}{\cfs\xi {n-1} }) =\cfs\xi n.\]
Otherwise, $\|\be^*\|<\|\be\|< \| \zeta \| $, so $\|\be^*\|\leq \|\zeta\|-2$.
The induction hypothesis yields $\beta^*\leq \cfs\xi {n-2} < \cfs\xi {n-1}$, and by Lemma \ref{lemmBoundMC}, $\beta < \fs{\check\al} {\cfs\xi{n-1}} $.
It follows that
\[\zeta=\gt(\beta) < \gt\big ( \fs{\check\al} {\cfs\xi{n-1}}\big ) = \cfs\xi n.\]

\end{Cases}
\end{Cases}

\item ($\zeta =\gt(\be)$ with $\be>\al$).
By Lemma \ref{lemmBackwardsExt}, $\zeta\leq \cfs\xi 1 \leq \cfs \xi n$.

\end{Cases}
\end{Cases}

\item ($\xi\geq \Om$).
Write $\zeta= \Om^{\zeta_0}\zeta_1 + \zeta_2$ and $\xi= \Om^{\xi_0}\xi_1+\xi_2$ in $\Om$-normal form, let $\tau=\tau(\xi)$, and consider the following cases.
\begin{Cases}
\item ($\xi_2>0$).
Then if $\Om^{\zeta_0}\zeta_1 < \Om^{\xi_0}\xi_1$, we have $\zeta<\fs\xi{\cfs\tau n} =\cfs\xi n$ regardless of $\cfs\tau n$, so we assume $\Om^{\zeta_0}\zeta_1 = \Om^{\xi_0}\xi_1$.
But then we may apply the induction hypothesis, since $\zeta< \xi $ yields $\zeta_2< \xi_2 $ and thus $\zeta_2 \leq\cfs{\xi_2}{n}$, so that $\zeta< \Om^{\xi_0}\xi_1+\cfs{\xi_2}{n} = \cfs\xi{ n}$.

\item ($\xi_2 = 0$).
We divide into the following sub-cases.
\begin{Cases}

\item ($\xi_1 \in \rm Lim$).
Then, $\fs\xi{\delta} = \Om^{\xi_0}\delta$ for all $\delta$.
If $\zeta_0<\xi_0$, we automatically have $\zeta< \Om^{\xi_0}\cfs\tau n  =\cfs\xi n$.
So, assume that $\zeta_0=\xi_0$.
In this case, we have that $\zeta_1< \tau$, so that the induction hypothesis and the fact that $\|\zeta_1\|<n$ yield $\zeta_1\leq \cfs\tau {n-1}$, hence $\zeta_1+1 \leq \cfs\tau n$.
Since $\zeta_2<\Om^{\zeta_0}$, we have that
\[ \zeta=  \Om^{\zeta_0}\zeta_1 + \zeta_2 < \Om^{\zeta_0}(\zeta_1 + 1) \leq \Om^{\xi_0}\cfs\tau n = \fs\xi{\cfs\tau n} =\cfs\xi n .\]

\item ($\xi_1 =\eta+1$).
Note that this includes the case where $\xi_1=1$.
Since $\tau(\xi)<\Om$, we must have $\xi_0\in \rm Lim$.
In this case, $\fs\xi{\delta} = \Om^{\xi_0}\eta + \Om^{\fs{\xi_0}\delta}$ for all $\delta$.
Consider the following sub-cases.

\begin{Cases}

\item ($\zeta_0< {\xi_0} $).
In this case, the induction hypothesis yields $\zeta_0< \cfs{\xi_0}{  n}$, which in turn yields
$ \zeta < \Om^{\cfs{\xi_0}{ n}} \leq \cfs\xi{  n}$, since $\Om^{\cfs{\xi_0}{  n}}$ is additively indecomposable.

\item ($\zeta_0 \geq   {\xi_0} $).
Note that $\zeta_0 >  {\xi_0} $ implies that $\zeta>\xi$, so we must have $\zeta_0 =  {\xi_0} $, and similarly, $\eta>0 $.
We consider further sub-cases.
\begin{Cases}

\item ($\zeta_1<\eta$).
Then, $\zeta_2< \Om^{\zeta_0}$ yields
\[\zeta<\Om^{\zeta_0} (\zeta_1+1) \leq  \Om^{\xi_0}\eta <\cfs\xi{ n}.\]

\item ($\zeta_1 = \eta$).
Since $\zeta_2 < \Om^{\xi_0} $, the induction hypothesis yields $\zeta_2 <  \cfs{\Om^{\xi_0}}{  n} =  \Om^{\cfs{\xi_0}{  n }}$.
We conclude that
\[\zeta= \Om^{\xi_0}\eta+\zeta_2 < \Om^{\xi_0}\eta + \Om^{\cfs{\xi_0}{ n}} = \cfs \xi n.\]

\end{Cases}

\end{Cases}

\end{Cases}

\end{Cases}

\end{Cases}
We have now considered all cases, and conclude the proof.
\end{proof}

We are now ready to state our main technical result.

\begin{theorem}\label{theoMain}\
\begin{enumerate}

\item $(\vartheta(\ve_{\Om+1}),\cfs\cdot\cdot_\vartheta,\|\cdot\|_\vartheta)$ is a regular Cantorian Bachmann system.

\item $(\vartheta(\ve_{\Om+1}),\cfs\cdot\cdot_\nt,\|\cdot\|_\nt)$ is a regular Bachmann system.

\end{enumerate}

\end{theorem}

\begin{proof}
All required properties have been proven in the text (Lemmas \ref{lemCantorOne} and \ref{lemCantorTwo}, Theorem \ref{theoBach}, Theorem \ref{theoReg}).
\end{proof}

\section{The Hardy hierarchy}\label{sectHardy}

If $\Lambda$ is a countable ordinal, a Hardy hierarchy for $\Lambda$ is a family of functions $(H_\lambda)_{\lambda<\Lambda}$ on the natural numbers, where $H_\lambda(n)$ is intended to be increasing both on $n$ and, asymptotically, on $\lambda$.
As mentioned in the introduction, these functions are useful for majorizing the provably total computable functions of a theory $T$; for example, every provably total computable function of Peano arithmetic is eventually bounded by $H_{\ve_0}$.
However, this hierarchy is not determined by the order-type of $\ve_0$ alone: it requires a `natural' notation system, as well as fundamental sequences, for $\ve_0$.
More generally, Hardy hierarchies are defined with respect to systems of fundamental sequences.

\begin{definition}
Let $(\Lambda,\cfs\cdot\cdot)$ be a system of fundamental sequences.
We define the {\em Hardy hierarchy based on $\Lambda$} as the function
\[H_\cdot(\cdot)\colon \Lambda \times\mathbb N\to \mathbb N\]
defined recursively by
\begin{enumerate}

\item $H_0(n) = n$,

\item $H_{\al+1}(n) = H_{\al}(n+1)$, and

\item $H_{\al}(n) = H_{\cfs\al n}(n )$ for $\al\in \rm Lim$.

\end{enumerate}
\end{definition}

For $\gt \in \{\nt,\vartheta\}$, we denote by $H^{\gt}$ the Hardy functions based on the system $\big (\vt(\ve_{\Om+1}),\cfs\cdot\cdot_\gt\big )$.
The following is proven for any regular Bachmann system in \cite{BCW}, hence by Theorem \ref{theoMain}, it holds for $H^\nt$ and $H^\vartheta$.

\begin{theorem}
For $\gt \in \{\nt,\vartheta\}$:
\begin{enumerate}

\item If $n<m$ and $\al<\vartheta(\ve_{\Om+1})$, then $H^\gt_\al(n) < H^\gt_\al(m)$.

\item If $\al<\be$ and $n\geq \|\al\|$, then $H^\gt_\al(n) \leq H^\gt_\be(n)$.

\end{enumerate}
\end{theorem}

There are various other regularity properties that hold for the Hardy hierarchies of all (Cantorian) regular Bachmann systems, for which we refer the reader to \cite{BCW}.
We also obtain the following comparison of the two hierarchies.

\begin{lemma}\label{lemmHardyCompare}
If $\xi<\vartheta(\ve_{\Om+1})$ and $0<n<m$, then $H_\xi^\vartheta(n) < H_\xi^\nt(m)$.
\end{lemma}

\begin{proof}
We prove the more general claim that if $n<m$ and $\cfs\xi n_\vartheta < \zeta \leq \xi$, then $H_\xi^\vartheta(n)<H_\zeta^\nt(m) $.
Proceed by induction on $\xi$ with a secondary induction on $\zeta$.
First assume that $\zeta=\xi$.
If $\xi=0$ the claim is immediate, and if $\xi $ is a successor then
\[H_\xi^\nt(m) = H^\nt_{\xi-1}(m+1) \stackrel{\text{\sc ih}} > H^\vartheta_{\xi-1}(n+1) = H^\vartheta_{\xi}(n) . \]
If $\xi \in \rm Lim$, we have by Theorem \ref{theoLeqTV} that $ \cfs\xi m_\nt \geq \cfs\xi n_\vartheta$.
If $ \cfs\xi m_\nt = \cfs\xi n_\vartheta$ we apply the main induction hypothesis to $\cfs\xi n_\vartheta$ and otherwise we apply the secondary induction hypothesis to  $\cfs\xi m_\nt  $ to also conclude that $H^\nt_{\cfs\xi m_\nt}(m)  > H^\vartheta_{\cfs \xi n_\vartheta }(n)$.

Otherwise, $\cfs\xi n_\vartheta < \zeta < \xi $.
In this case we must have $\xi\in\rm Lim$.
By the Bachmann property for $\vartheta$ and Theorem \ref{theoLeqTV}, we obtain
\[\cfs\xi n_\vartheta \leq \cfs\zeta n_\vartheta \leq \cfs\zeta m_\nt<\zeta \leq \xi.\]
Let $x=0$ if $\zeta\in \rm Lim$ and $x=1$ otherwise.
We may apply the secondary induction hypothesis to $\cfs\zeta m_\nt$ to conclude that $H^\nt_{\cfs\zeta m_\nt}(m+x)  > H^\vartheta_{\cfs \xi n}(n) $, and therefore
\[H_\zeta^\nt(m) = H^\nt_{\cfs\zeta m_\nt}(m+x)  > H^\vartheta_{\cfs \xi n}(n) = H^\vartheta_{\xi}(n).\qedhere\]
\end{proof}

As an easy corollary, we obtain that $H_\xi^\vartheta(n) < H_{\xi+1}^\nt(n)$ for all $\xi$ and $n>0$.
This version will be useful below.
Since $\big (\vartheta(\ve_{\Om+1}),\cfs\cdot\cdot_\vartheta \big )$ is a regular Cantorian Bachmann system, we can use $H^\vartheta$ to bound functions defined by recursion on $ \vartheta(\ve_{\Om+1})$.

\begin{definition}
Let $(\Lambda,\cfs\cdot\cdot,\|\cdot\|)$ be a normed system.
A {\em primitive recursive presentation of $\Lambda$} is a pair $(\mathcal A,\prec)$ such that $\mathcal A\subseteq\mathbb N$, both $\mathcal A$ and ${\prec}\subseteq \mathcal A\times \mathcal A $ are primitive recursively decidable, and there is a bijection $\ulcorner\cdot\urcorner \colon \Lambda\to \mathcal A$ such that
\begin{enumerate}
\item $\eta<\la$ if and only if $\ulcorner \eta\urcorner \prec \ulcorner \la\urcorner$,

\item if $\|\eta\|<\|\la\|$, then $\ulcorner \eta\urcorner < \ulcorner \la \urcorner $, and

\item there is a primitive recursive function $h$ with $\ulcorner\la\urcorner \leq h(\|\lambda\|)$ for all $\lambda<\Lambda$.
\end{enumerate}
\end{definition}

Note that the second item refers to the ordering on the natural numbers, {\em not} the ordering $\prec$ on $\mathcal A$; in other words, ordinals with larger norms have larger codes.
For example, for $\Lambda=\vartheta(\ve_{\Om+1})$, we can let $\ulcorner \la\urcorner$ be the G\"odel code of a term denoting $\la$ as given by Corollary \ref{corNotation}.
If we represent terms as strings using suitable base $k$, then we have that $\ulcorner\la\urcorner\leq k^{ \|\la\|}$.
The relation $\prec$ can be computed primitive recursively from Lemma \ref{lemmThetaOrd} and properties of Cantor normal forms.

The following class of functions is defined by~Cichon et al.~\cite{BCW}.
These functions are also similar (but not identical) to the {\em descent recursive} functions of Friedman and Sheard~\cite{FriedmanSheard}.

\begin{definition}
Given a well order $(\mathcal A,{\prec})$ with $\mathcal A\subseteq \mathbb N$, we define the set of functions $\mathcal R (\mathcal A, {\prec})$ as those functions $f\colon \mathbb N^n\to \mathbb N$ of the form
\[f(\vec x) = \min \{y: g(\vec x,y ) \preccurlyeq g(\vec x,y+1)   \},\]
where $g\colon \mathbb N^{n+1}\to \mathcal A$ is primitive recursive and such that there exists $a_* \in \mathcal A$ such that for all $\vec x\in \mathbb N^n$, $g(\vec x,0)\preccurlyeq a_*$.
\end{definition}

\begin{theorem}\label{theoH}
Let $\gt \in \{\vartheta,\nt\}$ and $(\mathcal B,\prec)$ be a primitive recursive presentation of $ \big( \vartheta(\ve_{\Om+1}),\cfs \cdot\cdot_\gt,\|\cdot\|_\gt \big )$.
Then, for every $f\in \mathcal R (\mathcal B,\prec)$, there is $\al < \vartheta(\ve_{\Om+1})$ such that $f(\vec x) < H^\Theta_\al(\max\vec x) $ for all $\vec x$.
\end{theorem}

\begin{proof}
This is shown to hold for all regular Cantorian Bachmann systems in \cite[Theorem 2]{BCW}, hence by Theorem \ref{theoMain}, it is true of $\gt =\vartheta$.
Lemma \ref{lemmHardyCompare} implies that every function $H^\vartheta_\la$ is bounded by $H^\nt_{\la+1}$, so the theorem holds for $\gt=\nt$ as well.
\end{proof}

\section*{Concluding remarks}\label{secConc}

We have shown that Buchholz's system of fundamental sequences for $\vartheta$, as well as its variant, $\nt$, enjoy some elemental properties useful in proving that the associated Hardy hierarchy is well-behaved.
Although we have not focused on formal theories in this work, Theorem \ref{theoH} yields independence results for many theories of Bachmann-Howard strength, incluing the theory ${\sf ID}_1$ of non-iterated inductive definitions; it is known that all provably total computable functions are defined by recursion along $\vartheta(\ve_{\Om+1})$~\cite{Eguchi}, hence they are majorized by both $H^\vartheta_{ \vartheta(\ve_{\Om+1})}$ and $H^\nt_{ \vartheta(\ve_{\Om+1})}$.
Conversely, any computable function that grows faster than this is {\em not} provably total.
This provides a strategy for establishing the independence of $\Pi^0_2$ statements from such theories, and indeed the authors have used it to develop a Goodstein process independent of ${\sf ID}_1$~\cite{GoodsteinFast}.
We trust that this work will be useful in establishing many new independence results in this spirit.

\subsection*{Acknowledgements}

The authors wish to thank the anonymous referees for their useful comments, including various errors found in a previous version of this article.
This work was partially supported by the FWO-FWF Lead Agency Grant G030620N.

\bibliographystyle{plain}
\bibliography{biblio}

\end{document}